\documentclass{article}
\usepackage{amsmath}
\usepackage{amssymb}
\usepackage{amsthm}
\usepackage{bbm}
\usepackage{cite}
\usepackage{xcolor}
\usepackage{stmaryrd}
\SetSymbolFont{stmry}{bold}{U}{stmry}{m}{n}
\usepackage{euscript}
\usepackage[arrow,curve,matrix,arc,2cell,frame]{xy}
\UseAllTwocells
\usepackage[utf8]{inputenc}
\usepackage[pdftex,breaklinks,unicode]{hyperref}
\usepackage{slashed}
\DeclareFontFamily{U}{rsfs}{} 
\DeclareFontShape{U}{rsfs}{n}{it}{<->
rsfs10}{} \DeclareSymbolFont{mscr}{U}{rsfs}{n}{it}
\DeclareSymbolFontAlphabet{\scr}{mscr}
\def\mathscr{\scr}
\begin{document}
\def\e#1\e{\begin{equation}#1\end{equation}}
\def\ea#1\ea{\begin{align}#1\end{align}}
\def\eq#1{{\rm(\ref{#1})}}
\theoremstyle{plain}
\newtheorem{thm}{Theorem}[section]
\newtheorem{lem}[thm]{Lemma}
\newtheorem{prop}[thm]{Proposition}
\newtheorem{cor}[thm]{Corollary}
\newtheorem{prob}[thm]{Problem}
\newtheorem{conj}[thm]{Conjecture}
\theoremstyle{definition}
\newtheorem{dfn}[thm]{Definition}
\newtheorem{ex}[thm]{Example}
\newtheorem{rem}[thm]{Remark}
\numberwithin{figure}{section}
\numberwithin{equation}{section}
\def\dim{\mathop{\rm dim}\nolimits}
\def\codim{\mathop{\rm codim}\nolimits}
\def\vdim{\mathop{\rm vdim}\nolimits}
\def\Im{\mathop{\rm Im}\nolimits}
\def\det{\mathop{\rm det}\nolimits}
\def\Res{\mathop{\rm Res}\nolimits}
\def\Ker{\mathop{\rm Ker}}
\def\Coker{\mathop{\rm Coker}}
\def\SSym{\mathop{\rm SSym}\nolimits}
\def\Sym{\mathop{\rm Sym}\nolimits}
\def\Ext{\mathop{\rm Ext}\nolimits}
\def\cExt{\mathop{{\mathcal E}\mathit{xt}}\nolimits}
\def\Spec{\mathop{\rm Spec}}
\def\Perf{\mathop{\rm Perf}}
\def\Vect{\mathop{\rm Vect}}
\def\HSta{\mathop{\bf HSta}\nolimits}
\def\Iso{\mathop{\rm Iso}\nolimits}
\def\Aut{\mathop{\rm Aut}}
\def\End{\mathop{\rm End}}
\def\Ho{\mathop{\rm Ho}}
\def\PGL{\mathop{\rm PGL}\nolimits}
\def\GL{\mathop{\rm GL}\nolimits}
\def\SL{\mathop{\rm SL}}
\def\SU{\mathop{\rm SU}}
\def\Re{\mathop{\rm Re}}
\def\PD{\mathop{\rm PD}}
\def\sht{{\text{\rm s-t}}}
\def\td{\mathop{\rm td}\nolimits}
\def\ch{\mathop{\rm ch}\nolimits}
\def\inc{\mathop{\rm inc}\nolimits}
\def\ind{\mathop{\rm ind}\nolimits}
\def\Stab{\mathop{\rm Stab}\nolimits}
\def\supp{\mathop{\rm supp}}
\def\rank{\mathop{\rm rank}\nolimits}
\def\Hom{\mathop{\rm Hom}\nolimits}
\def\id{{\mathop{\rm id}\nolimits}}
\def\Id{{\mathop{\rm Id}\nolimits}}
\def\boo{{\mathbin{\mathbbm 1}}}
\def\modCQ{\mathop{\text{\rm mod-}\C Q}}
\def\alg{{\rm alg}}
\def\hom{{\rm hom}}
\def\pl{{\rm pl}}
\def\ran{{\rm an}}
\def\rst{{\rm st}}
\def\ss{{\rm ss}}
\def\top{{\rm top}}
\def\virt{{\rm virt}}
\def\inv{{\rm inv}}
\def\fund{{\rm fund}}
\def\coh{{\rm coh}}
\def\Top{{\mathop{\bf Top}\nolimits}}
\def\bs{\boldsymbol}
\def\ge{\geqslant}
\def\le{\leqslant\nobreak}
\def\O{{\mathcal O}}
\def\bA{{\mathbin{\mathbb A}}}
\def\bF{{\mathbin{\mathbb F}}}
\def\bG{{\mathbin{\mathbb G}}}
\def\bH{{\mathbin{\mathbb H}}}
\def\bL{{\mathbin{\mathbb L}}}
\def\P{{\mathbin{\mathbb P}}}
\def\K{{\mathbin{\mathbb K}}}
\def\R{{\mathbin{\mathbb R}}}
\def\bT{{\mathbin{\mathbb T}}}
\def\Z{{\mathbin{\mathbb Z}}}
\def\bP{{\mathbin{\mathbb P}}}
\def\Q{{\mathbin{\mathbb Q}}}
\def\N{{\mathbin{\mathbb N}}}
\def\C{{\mathbin{\mathbb C}}}
\def\CP{{\mathbin{\mathbb{CP}}}}
\def\KP{{\mathbin{\mathbb{KP}}}}
\def\RP{{\mathbin{\mathbb{RP}}}}
\def\fC{{\mathbin{\mathfrak C}\kern.05em}}
\def\fD{{\mathbin{\mathfrak D}}}
\def\fE{{\mathbin{\mathfrak E}}}
\def\fF{{\mathbin{\mathfrak F}}}
\def\A{{\mathbin{\cal A}}}
\def\acA{{\mathbin{\smash{\acute{\mathcal A}}}\vphantom{\mathcal A}}}
\def\G{{{\cal G}}}
\def\M{{\mathbin{\cal M}}}
\def\B{{\mathbin{\cal B}}}  
\def\cC{{\mathbin{\cal C}}}
\def\cD{{\mathbin{\cal D}}}
\def\cE{{\mathbin{\cal E}}}
\def\cF{{\mathbin{\cal F}}}
\def\cG{{\mathbin{\cal G}}}
\def\cH{{\mathbin{\cal H}}}
\def\cI{{\mathbin{\cal I}}}
\def\cJ{{\mathbin{\cal J}}}
\def\cK{{\mathbin{\cal K}}}
\def\cL{{\mathbin{\cal L}}}
\def\cB{{\mathbin{\cal B}}}
\def\bcM{{\mathbin{\bs{\cal M}}}}
\def\cN{{\cal N}}
\def\cP{{\mathbin{\cal P}}}
\def\cQ{{\mathbin{\cal Q}}}
\def\cR{{\mathbin{\cal R}}}
\def\cS{{\mathbin{\cal S}}}
\def\T{{{\cal T}\kern .04em}}
\def\cU{{\mathbin{\cal U}}}
\def\cV{{\mathbin{\cal V}}}
\def\cW{{\mathbin{\cal W}}}
\def\cX{{\cal X}}
\def\cY{{\cal Y}}
\def\cZ{{\cal Z}}
\def\cV{{\cal V}}
\def\cW{{\cal W}}
\def\sC{{{\mathscr C}}}
\def\sF{{{\mathscr F}}}
\def\sR{{{\mathscr R}}}
\def\sS{{{\mathscr S}}}
\def\sT{{{\mathscr T}}}
\def\sU{{{\mathscr U}}}
\def\sV{{{\mathscr V}}}
\def\sW{{{\mathscr W}}}
\def\b{{\mathfrak b}}
\def\fe{{\mathfrak e}}
\def\f{{\mathfrak f}}
\def\g{{\mathfrak g}}
\def\h{{\mathfrak h}}
\def\k{{\mathfrak k}}
\def\m{{\mathfrak m}}
\def\n{{\mathfrak n}}
\def\p{{\mathfrak p}}
\def\q{{\mathfrak q}}
\def\s{{\mathfrak s}}
\def\u{{\mathfrak u}}
\def\H{{\mathfrak H}}
\def\so{{\mathfrak{so}}}
\def\su{{\mathfrak{su}}}
\def\sp{{\mathfrak{sp}}}
\def\fW{{\mathfrak W}}
\def\fX{{\mathfrak X}}
\def\fY{{\mathfrak Y}}
\def\fZ{{\mathfrak Z}}
\def\bfb{{\bs{\mathfrak b}}}
\def\bfc{{\bs{\mathfrak c}}}
\def\bfd{{\bs{\mathfrak d}}}
\def\bfe{{\bs{\mathfrak e}}}
\def\bff{{\bs{\mathfrak f}}}
\def\bfg{{\bs{\mathfrak g}}}
\def\bfh{{\bs{\mathfrak h}}}
\def\bfU{{\bs{\mathfrak U}}}
\def\bfV{{\bs{\mathfrak V}}}
\def\bfW{{\bs{\mathfrak W}}}
\def\bfX{{\bs{\mathfrak X}}}
\def\bfY{{\bs{\mathfrak Y}}}
\def\bfZ{{\bs{\mathfrak Z}}}
\def\bE{{\bs E}}
\def\bM{{\bs M}}
\def\bN{{\bs N}}
\def\bO{{\bs O}}
\def\bQ{{\bs Q}}
\def\bS{{\bs S}}
\def\bU{{\bs U}}
\def\bV{{\bs V}}
\def\bW{{\bs W}\kern -0.1em}
\def\bX{{\bs X}}
\def\bY{{\bs Y}\kern -0.1em}
\def\bZ{{\bs Z}}
\def\al{\alpha}
\def\be{\beta}
\def\ga{\gamma}
\def\de{\delta}
\def\io{\iota}
\def\ep{\epsilon}
\def\la{\lambda}
\def\ka{\kappa}
\def\th{\theta}
\def\ze{\zeta}
\def\up{\upsilon}
\def\vp{\varphi}
\def\si{\sigma}
\def\om{\omega}
\def\De{\Delta}
\def\Ka{{\rm K}}
\def\La{\Lambda}
\def\Mu{{\rm M}}
\def\Nu{{\rm N}}
\def\Om{\Omega}
\def\Ga{\Gamma}
\def\Si{\Sigma}
\def\Th{\Theta}
\def\Up{\Upsilon}
\def\Chi{{\rm X}}
\def\Tau{{\rm T}}
\def\Nu{{\rm N}}
\def\pd{\partial}
\def\ts{\textstyle}
\def\st{\scriptstyle}
\def\sst{\scriptscriptstyle}
\def\w{\wedge}
\def\sm{\setminus}
\def\lt{\ltimes}
\def\bu{\bullet}
\def\sh{\sharp}
\def\di{\diamond}
\def\he{\heartsuit}
\def\od{\odot}
\def\op{\oplus}
\def\ot{\otimes}
\def\hot{\mathbin{\hat\otimes}}
\def\bt{\boxtimes}
\def\bp{\boxplus}
\def\ov{\overline}
\def\bigop{\bigoplus}
\def\bigot{\bigotimes}
\def\tr{\blacktriangle}
\def\iy{\infty}
\def\es{\emptyset}
\def\ra{\rightarrow}
\def\rra{\rightrightarrows}
\def\Ra{\Rightarrow}
\def\Longra{\Longrightarrow}
\def\ab{\allowbreak}
\def\longra{\longrightarrow}
\def\hookra{\hookrightarrow}
\def\dashra{\dashrightarrow}
\def\lb{\llbracket}
\def\rb{\rrbracket}
\def\ha{{\ts\frac{1}{2}}}
\def\t{\times}
\def\pr{\preceq}
\def\tl{\trianglelefteq}
\def\ci{\circ}
\def\ti{\tilde}
\def\ac{\acute}
\def\gr{\grave}
\def\d{{\rm d}}
\def\an#1{\langle #1 \rangle}
\def\ban#1{\bigl\langle #1 \bigr\rangle}
\title{The Pandharipande--Thomas rationality conjecture for superpositive curve classes on projective complex 3-manifolds}
\author{Reginald Anderson and Dominic Joyce}
\date{}
\maketitle
\begin{abstract}
Let $X$ be a projective complex 3-manifold. An effective curve class $\be\in H_2(X,\Z)$ is called {\it positive\/} if $c_1(X)\cdot\be>0$, and {\it superpositive\/} if all the effective summands of $\be$ are positive. If $X$ is Fano then all effective classes are superpositive. The second author \cite{Joyc6} developed a theory of enumerative invariants in abelian categories and wall-crossing formulae. We use this  to prove conjectures by Pandharipande and Thomas \cite{PaTh2,Pand} on the rationality and poles of generating functions of Pandharipande--Thomas invariants of $X$ with descendent insertions, for superpositive curve classes.
\end{abstract}

{\baselineskip 11pt plus 1pt
\setcounter{tocdepth}{2}
\tableofcontents
\baselineskip 12pt}

\section{Introduction}
\label{pt1}

Let $X$ be a smooth, connected projective 3-fold over $\C$ throughout. We will prove the rationality conjecture for generating functions of descendent Pand\-hari\-pan\-de--Thomas invariants, Conjecture \ref{pt1conj1}(a),(b) below, for {\it superpositive\/} curve classes $\be$. The proof uses the second author's theory of enumerative invariants in abelian categories and wall-crossing \cite{Joyc6}. See \S\ref{pt2} for more details on the following definitions and notation.

\begin{dfn}
\label{pt1def1}
Let $A_1^\alg(X)$ denote the group of algebraic $1$-cycles on $X$ modulo algebraic equivalence. A class $\be$ in either $H_2(X,\Z)$ or $A_1^\alg(X)$ is called an {\it effective curve class\/} if it is represented by a nonzero sum of algebraic curves in $X$. If $\be,\ga$ are effective curve classes, we call $\ga$ a {\it factor\/} of $\be$ if either $\ga=\be$, or $\be=\ga+\de$ for some effective curve class $\de$. An effective curve class $\be$ has only finitely many factors. We call $\be$ {\it irreducible\/} if the only factor of $\be$ is $\be$ itself. We call $\be$ {\it positive\/} if $c_1(X)\cdot\be>0$, and {\it superpositive\/} if every factor of $\be$ is positive.
\end{dfn}

If $X$ is Fano then every effective curve class on $X$ is superpositive. The next definition follows Pandharipande--Thomas~\cite{PaTh1}.

\begin{dfn}
\label{pt1def2}
A {\it Pandharipande--Thomas stable pair\/} $(F,s)$ on $X$ is a pure 1-dimensional coherent sheaf $F$ on $X$ together with a section $s:\O_X\ra F$ with $0$-dimensional cokernel. We allow the case that $F=s=0$.
\end{dfn}

For an effective curve class $\be$ and $n\in\Z$, let $P_n(X,\be)$ denote the moduli scheme of Pandharipande--Thomas stable pairs $(F,s)$ with curve class $[\supp F]=\be$ (with multiplicity) and Euler characteristic $\chi(F)=n$. As in \cite{PaTh1} it is a proper $\C$-scheme with perfect obstruction theory, and carries a virtual class
\begin{equation*}
[P_n(X,\be)]_\virt\in H_{2c_1(X)\cdot\be}(P_n(X,\be),\Z).
\end{equation*}
We can do this with $\be$ in $A_1^\alg(X)$ or $H_2(X,\Z)$; when we want to distinguish these we write $P^\alg_n(X,\be)$ when $\be\in A_1^\alg(X)$ and $P^\hom_n(X,\be)$ when~$\be\in H_2(X,\Z)$.

The usual Pandharipande--Thomas descendent insertions in $H^*(P_n(X,\be),\Q)$ are denoted $\tau_k(\eta)$ for $k\ge 0$ and $\eta\in H^*(X,\Q)$, as in \S\ref{pt21}. For a product of insertions $\prod_{i=1}^m\tau_{k_i}(\eta_i)$ of total degree $2c_1(X)\cdot\be$, we call
\begin{equation*}
\ts PT_{\be,n}\bigl(\prod_{i=1}^m\tau_{k_i}(\eta_i)\bigr)=\ts\bigl(\prod_{i=1}^m\tau_{k_i}(\eta_i)\bigr)\cdot [P_n(X,\be)]_\virt
\end{equation*}
a {\it Pandharipande--Thomas invariant}, and we write
\begin{equation*}
\ts PT_\be\bigl(\prod_{i=1}^m\tau_{k_i}(\eta_i),q\bigr)=\sum_{n\in\Z}PT_{\be,n}\bigl(\prod_{i=1}^m\tau_{k_i}(\eta_i)\bigr)q^n.
\end{equation*}

We call $PT_\be\bigl(\prod_i\tau_{k_i}(\eta_i),q\bigr)$ the {\it descendent generating series\/} in class $\be$ with insertions $\prod_i\tau_{k_i}(\eta_i)$. It packages the virtual intersection numbers over the stable-pair moduli spaces $P_n(X,\be)$ with fixed curve class $\be$ and varying Euler characteristic $n$; since $P_n(X,\be)=\es$ for $n\ll 0$, it lies in $\Q[[q]][q^{-1}]$.

If a linear algebraic $\C$-group $G$ acts on $X$, we use the analogous notation $PT^G_\be(-,q)$ for the generating function of $G$-equivariant invariants.

Pandharipande--Thomas invariants have been extensively studied, see for example \cite{MOOP,Pand,PaPi1,PaPi2,PaTh1,PaTh2,PaTh3,PaTh4,StTh,Toda1,Toda2,Toda3,Toda4}. Here is an important conjecture on their structure (usually stated for $\be\in H_2(X,\Z)$, not $\be\in A_1^\alg(X)$). Part (a) is Pandharipande--Thomas \cite[Conj.~1]{PaTh2}, and (d) is proposed in the toric case with $G=\bG_m^3$ in \cite[Ex.~6.4]{PaTh2}. Parts (b),(c),(e),(f) come from Pandharipande~\cite[Conj.s 4, 5]{Pand}.

\begin{conj}
\label{pt1conj1}
{\bf(a)} The descendent generating series $PT_\be\bigl(\ts\prod_{i=1}^m\tau_{k_i}(\eta_i),q\bigr)$ is the Laurent expansion in $q$ of a rational function $F(q)\in\Q(q)$. 
\smallskip

\noindent{\bf(b)} The poles of\/ $F(q)$ occur only at\/ $q=0$ and at roots of unity.
\smallskip

\noindent{\bf(c)} The rational function $PT_\be\bigl(\ts\prod_{i=1}^m\tau_{k_i}(\eta_i),q\bigr)$ satisfies
\begin{equation*}
PT_\be\bigl(\ts\prod_{i=1}^m\tau_{k_i}(\eta_i),q^{-1}\bigr)=
(-1)^{\sum_{i=1}^mk_i}q^{-c_1(X)\cdot\be}PT_\be\bigl(\ts\prod_{i=1}^m\tau_{k_i}(\eta_i),q\bigr).
\end{equation*}

\noindent{\bf(d)} Suppose a linear algebraic $\C$-group $G$ acts on $X,$ and acts trivially on $H_2(X,\Z)$ and\/ $A_1^\alg(X),$ and take the $\eta_i$ to lie in $H^*_G(X,\Q)$. Then the $G$-equivariant descendent generating series $PT^G_\be\bigl(\ts\prod_{i=1}^m\tau_{k_i}(\eta_i),q\bigr)$ is the Laurent expansion in $q$ of a rational function $F^G(q)\in H^*_G(*,\Q)(q)$. 
\smallskip

\noindent{\bf(e)} The poles of\/ $F^G(q)$ occur only at\/ $q=0$ and at roots of unity.
\smallskip

\noindent{\bf(f)} The analogue of\/ {\bf(c)} holds for $PT^G_\be\bigl(\ts\prod_{i=1}^m\tau_{k_i}(\eta_i),q\bigr)$.
\end{conj}

Conjecture \ref{pt1conj1}(a),(c) are proved for Calabi--Yau 3-folds $X$ with $\prod_{i=1}^m\tau_{k_i}(\eta_i)\ab =1$ by Bridgeland \cite[Th.~1.1]{Brid} and Toda \cite[Cor.~1.3]{Toda3}. Parts (a), and (d) with $G=\bG_m^3$, are proved for toric smooth projective 3-folds $X$ by Pandharipande--Pixton \cite[Th.~1]{PaPi1}. Part (a) is proved for $X$ a Fano or Calabi--Yau complete intersection in a product of projective spaces, with constraints on the insertions $\tau_{k_i}(\eta_i)$, by Pandharipande--Pixton~\cite[Th.~1]{PaPi2}. Parts (a),(c) for general $X$ with `semi-Fano' curve classes $\be$ and `primary insertions' (that is, requiring $k_i=0$ in all $\tau_{k_i}(\eta_i)$) follow from Pardon~\cite[Th.~1.7]{Pard}.

Although this is often not stated, whenever the authors above prove parts (a) or (d), they implicitly prove (b) or (e) as well because of the method used to deduce rationality, see for instance Toda \cite[Proof of Lem.~4.6]{Toda2} and~\S\ref{pt34}.

Here is our first main result. It follows from Theorem \ref{pt1thm2} below.

\begin{thm}
\label{pt1thm1} 
Conjecture\/ {\rm\ref{pt1conj1}(a),(b),(d),(e)} hold when $\be$ is a \begin{bfseries}superpositive\end{bfseries} effective curve class, in the sense of Definition\/~{\rm\ref{pt1def1}}.
\end{thm}

Unfortunately we do not prove Conjecture \ref{pt1conj1}(c),(f). This is because of certain technical limitations in the theory of \cite{Joyc6}, discussed in Remark~\ref{pt2rem2}(c).

Ivan Karpov and Miguel Moreira \cite{KaMo2} have completely independently, and more-or-less simultaneously, proved Conjecture \ref{pt1conj1}(a)--(c) when $\be$ is a superpositive curve class for a threefold $X$ satisfying $H^{p,0}(X)=0$ for $p=1,2,3$. They do this using their beautiful paper \cite{KaMo1}, which is roughly a K-theory analogue of the first version of \cite{Joyc6}. As their theory lacks the technical limitations mentioned above, they can also prove Conjecture \ref{pt1conj1}(c) in the cases where their theory applies.

Our proof actually establishes a homology-valued refinement of Theorem \ref{pt1thm1}, before applying descendent insertions. Write $\M$ for the higher $\C$-stack of objects in $D^b\coh(X)$. Regarding a stable pair $\O_X\,{\buildrel s\over\longra}\, F$ as a complex in $D^b\coh(X)$ with $\O_X$ in degree $-1$ and $F$ in degree 0 induces a morphism $P^\alg_n(X,\be)\ra\M$. So we will regard the virtual class $[P^\alg_n(X,\be)]_\virt$ as lying in~$H_{2c_1(X)\cdot\be}(\M,\Q)$.

There is a decomposition $\M=\coprod_{\al\in K_0^\sht(X)}\M_\al$, where $K_0^\sht(X)$ is the $0^{\rm th}$ {\it semi-topological K-theory group\/} of $X$, as in \cite{FHW,FrWa1,FrWa2,FrWa3}, and each $\M_\al$ is connected. In \S\ref{pt21} we define a group morphism $\up:\Z\op A_1^\alg(X)\op\Z\ra K_0^\sht(X)$ such that $P^\alg_n(X,\be)$ maps to the component $\M_{\up(1,\be,n)}$ in $\M$, so that $[P^\alg_n(X,\be)]_\virt$ lies in~$H_{2c_1(X)\cdot\be}(\M_{\up(1,\be,n)},\Q)\subset H_{2c_1(X)\cdot\be}(\M,\Q)$.

Taking direct sum with a fixed complex in class $\up(1,\be,n)$ yields an $\bA^1$-homotopy equivalence $\M_0\ra\M_{\up(1,\be,n)}$, so $H_*(\M_{\up(1,\be,n)},\Q)\cong H_*(\M_0,\Q)$. Thus we can regard $[P^\alg_n(X,\be)]_\virt$ as lying in $H_{2c_1(X)\cdot\be}(\M_0,\Q)$, a vector space independent of $n\in\Z$. We write it as $[P^\alg_n(X,\be)]^0_\virt$ to indicate that we have moved it to $H_*(\M_0,\Q)$. This is necessary to state Theorem \ref{pt1thm2} below, as without it $[P^\alg_n(X,\be)]_\virt$ would lie in a vector space depending on $n$, and \eq{pt1eq1} would make no sense. This is also our reason for sometimes taking curve classes in $A_1^\alg(X)$ rather than $H_2(X,\Z)$, as $P^\hom_n(X,\be)$ for $\be\in H_2(X,\Z)$ may map to several components~$\M_\al$.

Now suppose a linear algebraic group $G$ acts on $X$, and acts trivially on $A_1^\alg(X)$. Then as in \S\ref{pt27}--\S\ref{pt28}, we can promote Pandharipande--Thomas virtual classes $[P^\alg_n(X,\be)]^0_\virt$ to $G$-equivariant homology $H_{2c_1(X)\cdot\be}^G(\M_0,\Q)$. Surprisingly, for reasons which will appear in Proposition \ref{pt3prop1} and Remark \ref{pt3rem1} below, the authors do {\it not\/} expect the rationality statement in Theorem \ref{pt1thm2}(a) to be true for $G$-equivariant virtual classes~$[P^\alg_n(X,\be)]^{0,G}_\virt$.

Because of this, in \S\ref{pt27} below for each $N\ge 0$ we define a {\it truncated\/} version of $G$-equivariant homology $H_*^{G,\le N}(S,\Q)$ for $\C$-stacks $S$. It has a surjective morphism $H_*^G(S,\Q)\twoheadrightarrow H_*^{G,\le N}(S,\Q)$ which roughly speaking quotients $H_*^G(S,\Q)$ by the action of $H^{>N}_G(*,\Q)$. Then in Theorem \ref{pt1thm2}(b) we prove rationality for generating functions of truncated classes $[P^\alg_n(X,\be)]^{0,G,\le N}_\virt$. This is enough to deduce Conjecture \ref{pt1conj1}(d),(e), as for any fixed insertion $\prod_{i=1}^m\tau_{k_i}(\eta_i)$ the $G$-equivariant Pandharipande--Thomas invariant $\bigl(\prod_{i=1}^m\tau_{k_i}(\eta_i)\bigr)\cdot[P^\alg_n(X,\be)]^{0,G}_\virt$ factors via $[P^\alg_n(X,\be)]^{0,G,\le N}_\virt$ for large enough~$N$.

Here is our second main result, which will be proved in \S\ref{pt3}. We describe $H_*(\M_0,\Q)$ explicitly in \eq{pt2eq8} below.

\begin{thm}
\label{pt1thm2}
{\bf(a)} Let\/ $X$ be a smooth, connected projective $3$-fold over\/ $\C,$ and\/ $\be\in A_1^\alg(X)$ be a superpositive curve class. Then the formal power series
\e
\sum_{n\in\Z}[P_n^\alg(X,\be)]^0_\virt q^n\in H_{2c_1(X)\cdot\be}(\M_0,\Q)[[q]][q^{-1}]
\label{pt1eq1}
\e
is the Laurent expansion of a rational function $F(q)\in H_{2c_1(X)\cdot\be}(\M_0,\Q)(q),$ which has poles only at\/ $q=0$ and at roots of unity.
\smallskip

\noindent{\bf(b)} Now suppose a linear algebraic $\C$-group $G$ acts on $X,$ and acts trivially on $A_1^\alg(X)$. Then for each $N\ge 0,$ the formal power series
\e
\sum_{n\in\Z}[P_n^\alg(X,\be)]^{0,G,\leq N}_\virt q^n\in H_{2c_1(X)\cdot\be}^{G,\le N}(\M_0,\Q)[[q]][q^{-1}]
\label{pt1eq2}
\e
is the expansion of a rational function $F^{G,\le N}(q)\in H_{2c_1(X)\cdot\be}^{G,\le N}(\M_0,\Q)(q),$ which has poles only at\/ $q=0$ and at roots of unity.
\end{thm}

To deduce Theorem \ref{pt1thm1} from Theorem \ref{pt1thm2}, note that the generating function $PT_\be\bigl(\ts\prod_{i=1}^m\tau_{k_i}(\eta_i),q\bigr)$ is obtained by evaluating a cohomology class $\prod_{i=1}^m\tau_{k_i}(\eta_i)$ on the series \eq{pt1eq1}. The isomorphisms $H_*(\M_{\up(1,\be,n)},\Q)\cong H_*(\M_0,\Q)$ above identify the cohomology classes $\tau_{k_i}(\eta_i)$ on $\M_{\up(1,\be,n)}$ and $\M_0$ for $k_i\ge -1$, so replacing $[P^\alg_n(X,\be)]_\virt$ by $[P^\alg_n(X,\be)]^0_\virt$ does not change things. Thus Conjecture \ref{pt1conj1}(a),(b) for $\be\in A_1^\alg(X)$ follow from Theorem \ref{pt1thm2}(a). To prove Conjecture \ref{pt1conj1}(a),(b) for $\be\in H_2(X,\Z)$, we sum Conjecture \ref{pt1conj1}(a),(b) for $\ga\in A_1^\alg(X)$ over the finitely many effective $\ga\in A_1^\alg(X)$ with $\Pi_\alg^\hom(\ga)=\be$, as~in~\eq{pt2eq1}.

For the $G$-equivariant case, note that Theorem \ref{pt1thm1}(b) gives rationality in $N$-{\it truncated\/} $G$-equivariant homology $H_*^{G,\le N}(-,\Q)$. Evaluating a $G$-equivariant cohomology class $\prod_{i=1}^m\tau_{k_i}(\eta_i)$ on \eq{pt1eq2} gives a rational function in
\begin{equation*}
H^{G,\le N}_{2c_1(X)\cdot\be - \deg(\prod_{i=1}^m\tau_{k_i}(\eta_i))}(*,\Q)[[q]][q^{-1}].
\end{equation*}
If $N\ge \deg(\prod_{i=1}^m\tau_{k_i}(\eta_i))-2c_1(X)\cdot\be$ then
\begin{align*}
H^{G,\le N}_{2c_1(X)\cdot\be-\deg(\prod_{i=1}^m\tau_{k_i}(\eta_i))}(*,\Q)&=
H^G_{2c_1(X)\cdot\be-\deg(\prod_{i=1}^m\tau_{k_i}(\eta_i))}(*,\Q)\\
&\cong H_G^{\deg(\prod_{i=1}^m\tau_{k_i}(\eta_i))-2c_1(X)\cdot\be}(*,\Q).
\end{align*}
Hence Conjecture \ref{pt1conj1}(d),(e) follow from Theorem \ref{pt1thm2}(b) as for Conjecture \ref{pt1conj1}(a),\ab (b), provided we take $N$ large enough.

If the morphisms \eq{pt2eq5} below are not injective after $-\ot_\Z\Q$ then Theorem \ref{pt1thm2}(a) is stronger than Conjecture \ref{pt1conj1}(a),(b), as by \eq{pt2eq8} below the Pand\-hari\-pan\-de--Thomas invariants $PT^\alg_{\be,n}\bigl(\ts\prod_{i=1}^m\tau_{k_i}(\eta_i)\bigr)$ taken over all $k_i\ge -1$ and $\eta_i$ do not determine $[P_n^\alg(X,\be)]_\virt$ as an element of~$H_*(\M_{\up(1,\be,n)},\Q)$.

Here is an outline of the proof of Theorem \ref{pt1thm2}:
\begin{itemize}
\setlength{\itemsep}{0pt}
\setlength{\parsep}{0pt}
\item[(i)] In the second version of \cite{Joyc6}, the second author will prove a wall-crossing formula relating classes $\Pi^\pl_*\bigl([P_n(X,\be)]_\virt\bigr)$ in \eq{pt2eq11} with Donaldson--Th\-omas invariants $[\M_{(\be,n)}^\ss(\mu)]_\inv$ counting 1-dimensional $\mu$-semistable coherent sheaves $F$ on $X$ with $\lb F\rb=(\be,n)$, for superpositive $\be$. This is part of a much larger theory \cite{Joyc6} of invariants and wall-crossing formulae for semistable objects in abelian categories.

The wall-crossing formula is written using a Lie bracket on the homology $H_*(\M^\pl)$ of the `projective linear' moduli stack $\M^\pl$ of objects in $D^b\coh(X)$. This Lie bracket is defined using a vertex algebra structure on the homology $H_*(\M)$ of the ordinary moduli stack $\M$ of objects in $D^b\coh(X)$. These structures were discovered by the second author~\cite{Joyc5}.
\item[(ii)] If the stability condition $\mu$ is defined using K\"ahler class $\om=c_1(L)$ for $L\ra X$ an ample line bundle, then tensor product by $L$ induces an isomorphism $\M_{(\be,n)}^\ss(\mu)\ra \M_{(\be,n+c_1(L)\cdot\be)}^\ss(\mu)$. Thus, the $[\M_{(\be,n)}^\ss(\mu)]_\inv$ have a periodicity property in $n$, made precise in Proposition \ref{pt3prop1} below. 
\item[(iii)] We will prove that for superpositive $\be\in A_1^\alg(X)$, there exist $N\in\Z$, $d\ge 1$, and polynomials $P_j(n)\in H_{2c_1(X)\cdot\be}(\M_0,\Q)[n]$ for $1\le j\le d$, with
\begin{equation*}
[P_n^\alg(X,\be)]^0_\virt=P_j(n)\quad\text{if $n\ge N$ and $n\equiv j\mod d$.}
\end{equation*}
We do this using the wall-crossing formula in (i), a technique for lifting $\Pi^\pl_*\bigl([P_n(X,\be)]_\virt\bigr)$ to $[P_n(X,\be)]_\virt$ in \S\ref{pt26}, explicit computations in vertex algebras, and induction on the number of factors of $\be$. Here the periodicity property in $n$ mod $d$ is deduced from the periodicity property of the $[\M_{(\be,n)}^\ss(\mu)]_\inv$ in (ii).
\item[(iv)] Theorem \ref{pt1thm2}(a) follows from (iii) and $P_n(X,\be)=\es$ if $n\ll 0$. 
\item[(v)] The $G$-equivariant analogues hold, provided we work in $H_*^{G,\le N}(\cdots)$. 
\end{itemize}

We are broadly following a well-known method: the proofs of Conjecture \ref{pt1conj1}(a)--(c) for Calabi--Yau 3-folds by Bridgeland \cite{Brid} and Toda \cite{Toda3,Toda4} use this strategy with the Joyce--Song wall-crossing formula \cite{JoSo} for Donaldson--Thomas invariants of Calabi--Yau 3-folds.

There are two main differences between our approach and \cite{Brid,Toda3,Toda4}: firstly, in the superpositive case the invariants are homology classes rather than rational numbers, and the Lie bracket, involving vertex algebras, is far more complicated. Secondly, we use a different change of stability condition to \cite{Brid,Toda3,Toda4}, which unfortunately does not allow us to prove Conjecture \ref{pt1conj1}(c),(f). This is because of certain technical limitations in the theory of \cite{Joyc6}, which mean it currently cannot be applied to the Bridgeland--Toda set up; see Remark \ref{pt2rem2}(c) on this.

A sequel by the first author \cite{Ande} uses the wall-crossing formulae \eq{pt2eq40}--\eq{pt2eq42} below to compute examples of Pandharipande--Thomas invariants from one-dimensional Donaldson--Thomas invariants, and vice versa.
\medskip

\noindent{\it Acknowledgements.} The authors would like to thank H\"ulya Arg\"uz and Pierrick Bousseau for useful conversations, and Ivan Karpov and Miguel Moreira for helpful comments, and for generously agreeing to the simultaneous arXiv release of their parallel paper \cite{KaMo2}, and some referees. For the purpose of open access, the authors have applied a CC BY public copyright licence to any Author Accepted Manuscript (AAM) version arising from this submission. The second author was supported during this research by EPSRC grant EP/X040674/1.

\section{Background material}
\label{pt2} 

We will assume the reader is already familiar with $\C$-schemes $X$ and the abelian category $\coh(X)$ of coherent sheaves on $X$, as in Hartshorne \cite{Hart}, with Artin $\C$-stacks as in G\'omez \cite{Gome}, Olsson \cite{Olss} and Laumon--Moret-Bailly \cite{LaMo}, with triangulated categories and derived categories $D^b\A$ as in Gelfand--Manin \cite{GeMa} and derived categories of coherent sheaves $D^b\coh(X)$ as in Huybrechts \cite{HuLe}, and with Gieseker (semi)stability of coherent sheaves and moduli schemes of (semi)stable sheaves as in Huybrechts--Lehn \cite{HuLe} and Gieseker~\cite{Gies}.

\subsection{Pandharipande--Thomas theory}
\label{pt21}

Throughout this section let $X$ be a smooth, connected projective 3-fold over $\C$.

\begin{dfn}
\label{pt2def1}
Let $Z_1(X)$ be the group of algebraic $1$-cycles on $X$. We write
\begin{equation*}
A_1^\alg(X)=Z_1(X)/{\sim_\alg}
\end{equation*}
for the abelian group of algebraic 1-cycles modulo algebraic equivalence. There is a natural cycle-class morphism
\begin{equation*}
\Pi_\alg^\hom:A_1^\alg(X)\longrightarrow H_2(X,\Z).
\end{equation*}
Whenever $\be\in A_1^\alg(X)$ and $\alpha\in H^2(X,\Q)$, we write
$\alpha\cdot\be$ for $\alpha\cdot\Pi_\alg^\hom(\be)$.

Note that $A_1^\alg(X)$ is different from the Chow homology group $CH_1(X)$ of algebraic 1-cycles on $X$ modulo {\it rational\/} equivalence, which has a surjective morphism $CH_1(X)\twoheadrightarrow A_1^\alg(X)$. Also $A_1^\alg(X)$ is discrete, in contrast to $CH_1(X)$. 

If $F$ is a pure 1-dimensional coherent sheaf on $X$ then the support $\supp F$, taken with multiplicity, is an algebraic 1-cycle, and has a class $[\supp F]$ in $A_1^\alg(X)$ or $H_2(X,\Z)$, which is an effective curve class if $F\ne 0$. For the $H^2(X,\Z)$ version we have $[\supp F]=\PD(c_2(F))$, the Poincar\'e dual of the second Chern class $c_2(F)$ in $H^4(X,\Z)$. Define the {\it class\/} of $F$ in either $A_1^\alg(X)\op\Z$ or $H_2(X,\Z)\op\Z$ to be $\lb F\rb=([\supp F],n)$, where $n=\chi(F)=\dim H^0(F)-\dim H^1(F)$ is the holomorphic Euler characteristic of $F$.
\end{dfn}

\begin{dfn}
\label{pt2def2}
As in Pandharipande--Thomas \cite[\S 2]{PaTh1}, there is a projective moduli $\C$-scheme $P_n(X,\be)$ whose $\C$-points $[F,s]$ are isomorphism classes of stable pairs $(F,s)$ with $\lb F\rb=(\be,n)$. Here either $(\be,n)\in A_1^\alg(X)\op\Z$ or $(\be,n)\in H_2(X,\Z)\op\Z$. When we want to distinguish the two we write $P_n^\alg(X,\be)$ for the former, and $P_n^\hom(X,\be)$ for the latter. For $\be\in H_2(X,\Z)$ we have
\e
P^\hom_n(X,\be)=\coprod_{\ga\in A_1^\alg(X):\Pi_\alg^\hom(\ga)=\be}P^\alg_n(X,\ga),
\label{pt2eq1}
\e
with only finitely many nonempty terms on the right hand side. If $P_n(X,\be)\ne\es$ then either $(\be,n)=(0,0)$, in which case $P_0(X,0)\cong\Spec\C$ is the single point $[0,0]$, or $\be$ is an effective curve class and $n\in\Z$. We have $P_n(X,\be)=\es$ if~$n\ll 0$.

There is a natural perfect obstruction theory $\phi:\cE^\bu\ra\bL_{P_n(X,\be)}$ on $P_n(X,\be)$ in the sense of Behrend--Fantechi \cite{BeFa}, with $\rank\cE^\bu=c_1(X)\cdot\be$. This is defined by regarding pairs $\O_X\,{\buildrel s\over\longra}\, F$ as objects $\cI^\bu$ in the derived category $D^b\coh(X)$, where $\O_X$ is in degree $-1$ and $F$ in degree 0, and $\det\cI^\bu\cong\O_X$. Then the obstruction theory is defined using the trace-free Ext complex $\cExt^\bu(\cI^\bu,\cI^\bu)_0$, and is natural for deformations of objects $\cI^\bu$ in $D^b\coh(X)$ with fixed determinant $\det\cI^\bu\cong\O_X$. Thus we have a virtual class
\begin{equation*}
[P_n(X,\be)]_\virt\in H_{2c_1(X)\cdot\be}(P_n(X,\be),\Z).
\end{equation*}
This is zero for dimensional reasons unless either $c_1(X)\cdot\be=0$ (the `Calabi--Yau case' \cite[\S 2.4]{PaTh1}) or $c_1(X)\cdot\be>0$ (the `Fano case'~\cite[\S 3.6]{PaTh1}).

There is a universal coherent sheaf $\fF\ra X\t P_n(X,\be)$, flat over $P_n(X,\be)$, and a universal section $\s:\O_{X\t P_n(X,\be)}\ra\fF$, such that the restriction of $(\fF,\s)$ to $X\t\{[(F,s)]\}$ is isomorphic to $(F,s)$ for all $\C$-points $[(F,s)]\in P_n(X,\be)$.

For all $k\in\N$ and $\eta\in H^l(X,\Q)$, define $\tau_k(\eta)\in H^{2k+l-2}(P_n(X,\be),\Q)$ by 
\e
\tau_k(\eta)=(\Pi_{P_n(X,\be)})_*\bigl(\Pi_X^*(\eta)\cup\ch_{2+k}(\fF)\bigr),
\label{pt2eq2}
\e
where $\Pi_X,\Pi_{P_n(X,\be)}$ are the projections from $X\t P_n(X,\be)$ to $X,P_n(X,\be)$. These are the usual Pandharipande--Thomas descendent insertions.

Suppose $m\ge 0$ and $k_i\in\N$, $\eta_i\in H^{l_i}(X,\Q)$ for $i=1,\ldots,m$ with $\sum_{i=1}^m(2k_i+l_i-2)=2c_1(X)\cdot\be$.  Then we define the {\it Pandharipande--Thomas invariant\/}
\e
PT_{\be,n}\bigl(\ts\prod_{i=1}^m\tau_{k_i}(\eta_i)\bigr)=\bigl(\ts\prod_{i=1}^m\tau_{k_i}(\eta_i)\bigr)\cdot[P_n(X,\be)]_\virt
\quad\text{in $\Q$.}
\label{pt2eq3}
\e
We combine these into a generating function
\e
PT_\be\bigl(\ts\prod_{i=1}^m\tau_{k_i}(\eta_i),q\bigr)=\sum_{n\in\Z}PT_{\be,n}\bigl(\ts\prod_{i=1}^m\tau_{k_i}(\eta_i)\bigr)q^n \quad\text{in $\Q[[q]][q^{-1}]$.}
\label{pt2eq4}
\e
\end{dfn}

\subsection{\texorpdfstring{Moduli stacks of objects in $D^b\coh(X)$}{Moduli stacks of objects in Dᵇcoh(X)}}
\label{pt22}

The first part of this section works for $X$ of any dimension~$m\ge 0$.

\begin{dfn}
\label{pt2def3}
Let $X$ be a smooth, connected, projective $\C$-scheme with
$\dim_\C X=m$. Write $\coh(X)$ for the abelian category of coherent
sheaves on $X$ and $D^b\coh(X)$ for its derived category. We identify their
Grothendieck groups
\begin{equation*}
K_0(\coh(X))=K_0(D^b\coh(X)).
\end{equation*}

We write $K_i^\sht(X)$, $i\ge 0$ for the {\it semi-topological K-theory\/} of $X$, as in Friedlander--Haesemeyer--Walker \cite{FHW,FrWa1,FrWa2,FrWa3}. This interpolates between the algebraic K-theory and the topological K-theory of $X$. There are natural morphisms
\e
\Pi^\sht_\top:K_i^\sht(X)\longra K^{-i}_\top(X^\ran)=:K^{-i}_\top(X)
\label{pt2eq5}
\e
to the topological complex K-theory $K^*_\top(-)$ of the underlying complex analytic space $X^\ran$ of $X$. By Bott periodicity $K^{2j}_\top(X)\cong K^0_\top(X)$ and $K^{2j+1}_\top(X)\cong K^1_\top(X)$ for all $j\in\Z$. Here $K_0^\sht(X)$ is the Grothendieck group of algebraic vector bundles modulo algebraic equivalence. 

There is a natural surjective morphism $K_0(\coh(X))\twoheadrightarrow K_0^\sht(X)$. For each object $E^\bu\in D^b\coh(X)$, we write $\lb E^\bu\rb\in K_0^\sht(X)$ for the image of $[E^\bu]\in K_0(\coh(X))$ under this morphism. The Chern character $\ch:K_0(\coh(X))\ra H^{\rm even}(X,\Q)$ factors as $K_0(\coh(X))\ab\twoheadrightarrow K_0^\sht(X)\ra H^{\rm even}(X,\Q)$.

Write $\M$ for the moduli stack of objects in $D^b\coh(X)$. It is a higher $\C$-stack in the sense of To\"en--Vezzosi \cite{Toen1,Toen2,ToVe1,ToVe2}, and exists by To\"en--Vaqui\'e \cite{ToVa}. The $\C$-points of $\M$ are isomorphism classes $[E^\bu]$ of objects $E^\bu\in D^b\coh(X)$. There is a natural decomposition $\M=\coprod_{\al\in K_0^\sht(X)}\M_\al$, where $\M_\al$ is the moduli stack of $E^\bu\in D^b\coh(X)$ with $\lb E^\bu\rb=\al$ in $K_0^\sht(X)$. Then $\M_\al$ is open and closed in $\M$, and furthermore $\M_\al$ is nonempty and connected for each $\al\in K_0^\sht(X)$. That is, the set of connected components $\pi_0(\M)$ of $\M$ is exactly $K_0^\sht(X)$. This follows from the definition of $K_0^\sht(X)$ and properties of perfect complexes.

As we explain in \S\ref{pt24}, there is a natural morphism $\Psi:[*/\bG_m]\t\M\ra\M$ which on $\C$-points acts by $\Psi_*:(*,[E^\bu])\mapsto[E^\bu]$, for all objects $E^\bu$ in $D^b\coh(X)$, and on isotropy groups acts by $\Psi_*:\Iso_{[*/\bG_m]\t\M}(*,[E^\bu])\cong\bG_m\t\Aut(E^\bu)\ra \Iso_\M([E^\bu])\cong\Aut(E^\bu)$ by $(\la,\mu)\mapsto \la\mu=(\la\cdot\id_{E^\bu})\ci\mu$ for $\la\in\bG_m$ and $\mu\in\Aut(E^\bu)$. Here $[*/\bG_m]$ is a group stack, and $\Psi:[*/\bG_m]\t\M\ra\M$ is an action of $[*/\bG_m]$ on $\M$, which is free on $\M\sm\{[0]\}$. We may take the quotient of $\M$ by $\Psi$ to get a stack $\M^\pl$, which we call the {\it projective linear moduli stack}, with projection $\Pi^\pl:\M\ra\M^\pl$, in a co-Cartesian square in the $\iy$-category $\HSta_\C$ of higher $\C$-stacks:
\begin{equation*}
\xymatrix@C=130pt@R=15pt{ *+[r]{[*/\bG_m]\t\M} \drtwocell_{}\omit^{}\omit{^{}} \ar[r]_(0.65){\Psi} \ar[d]^{\pi_\M} & *+[l]{\M} \ar[d]_{\Pi^\pl} \\
*+[r]{\M} \ar[r]^(0.35){\Pi^\pl} & *+[l]{\M^\pl.\!} }
\end{equation*}
This construction is known in the literature as {\it rigidification}, as in Abramovich--Olsson--Vistoli \cite{AOV} and Romagny \cite{Roma}, written $\M^\pl=\M\!\!\fatslash\,\bG_m$ in \cite{AOV,Roma}.

The splitting $\M=\coprod_{\al\in K_0^\sht(X)}\M_\al$ descends to $\M^\pl=\coprod_{\al\in K_0^\sht(X)}\M_\al^\pl$, with $\M_\al^\pl=\M_\al\!\!\fatslash\,\bG_m$ nonempty and connected.
\end{dfn}

\begin{dfn}
\label{pt2def4}
As in Simpson \cite{Simp} and Blanc \cite[\S 3.1]{Blan}, a higher $\C$-stack $S$ has a {\it topological realization\/} $S^\top$, which is a topological space natural up to homotopy equivalence. Topological realization gives a functor $(-)^\top:\Ho(\HSta_\C)\ab\ra\Top^{\bf ho}$ from the homotopy category of $\HSta_\C$ to the category $\Top^{\bf ho}$ of topological spaces with morphisms homotopy classes of continuous maps. 

Let $S$ be a higher $\C$-stack. We define the {\it homology $H_*(S)$ of\/ $S$ with coefficients in\/} $\Q$ to be $H_*(S)=H_*(S,\Q)=H_*(S^\top,\Q)$, the usual homology of the topological space $S^\top$ over $\Q$. Similarly we define the cohomology $H^*(S)=H^*(S,\Q)=H^*(S^\top,\Q)$. These are sometimes called the {\it Betti (co)homology}, to distinguish them from other (co)homology theories of stacks. 

If a linear algebraic $\C$-group $G$ acts on $S$, we can also define $G$-{\it equivariant\/} ({\it co\/}){\it homology\/} $H_*^G(S,\Q),H^*_G(S,\Q)$, which we discuss in \S\ref{pt27}. Note that our version of equivariant homology may be unfamiliar to some readers. 
\end{dfn} 

We will almost always take (co)homology of topological spaces or stacks {\it over the rationals\/} $\Q$, and when we omit the coefficient ring we mean it to be~$\Q$. 

Let $X,\M,\M_\al,\M^\pl,\M^\pl_\al$ be as in Definition \ref{pt2def3}. We will be interested in the $\Q$-homology groups
\e
H_*(\M,\Q)=\bigop_{\!\!\!\al\in K_0^\sht(X)\!\!\!}H_*(\M_\al,\Q),\quad H_*(\M^\pl,\Q)=\bigop_{\!\!\!\al\in K_0^\sht(X)\!\!\!}H_*(\M^\pl_\al,\Q).
\label{pt2eq6}
\e
We explain in \S\ref{pt24}--\S\ref{pt25} that by the second author \cite{Joyc5}, $H_*(\M)$ has the structure of a {\it graded vertex algebra}, and $H_*(\M^\pl)$ the structure of a {\it graded Lie algebra}.

For $\al\in K_0^\sht(X)$, pick $E^\bu\in D^b\coh(X)$ with $\lb E^\bu\rb=\al$. There is a stack morphism $\M_0\ra\M_\al$ mapping $[F^\bu]\mapsto[E^\bu\op F^\bu]$ on $\C$-points. This is an $\bA^1$-homotopy equivalence, with $\bA^1$-homotopy inverse $\M_\al\ra\M_0$ mapping $[F^\bu]\mapsto[E^\bu[1]\op F^\bu]$ on $\C$-points. Thus it induces an isomorphism
\e
H_*(\M_\al,\Q)\cong H_*(\M_0,\Q).
\label{pt2eq7}
\e
This is independent of the choice of $[E^\bu]\in\M_\al$, as $\M_\al$ is connected. Note that the analogue does {\it not\/} work for $\M_\al^\pl$, $\M_0^\pl$, as mapping $[F^\bu]\mapsto[E^\bu\op F^\bu]$ maps $\bG_m\id_{F^\bu}\mapsto \id_{E^\bu}\op \bG_m\id_{F^\bu}$, rather than $\bG_m\id_{F^\bu}\mapsto \bG_m(\id_{E^\bu}\op\id_{F^\bu})$, on isotropy groups.

Using work of Antieau--Heller \cite[Th.~2.3]{AnHe}, Blanc \cite[Th.~4.21]{Blan}, and Milnor--Moore \cite[App.]{MiMo}, the second author's PhD student Jacob Gross \cite[\S 4]{Gros1}, \cite[\S 4]{Gros2} shows that the homotopy groups $\pi_i(\M_\al^\top)$ have canonical isomorphisms $\pi_i(\M_\al^\top)\cong K_i^\sht(X)$ for all $\al\in K_0^\sht(X)$ and $i>0$, and for each $\al\in K_0^\sht(X)$ we have a canonical isomorphism
\e
H_*(\M_\al,\Q)\cong\SSym^*\bigl(\bigop\nolimits_{i\ge 1}K_i^\sht(X)\ot_\Z\Q\bigr).
\label{pt2eq8}
\e
Here $\SSym^*(-)=\bigop_{l\ge 0}\SSym^l(-)$ denotes the supersymmetric algebra of a $\Z$-graded $\Q$-vector space, and\/ $K_i^\sht(X)\ot_\Z\Q$ is graded of degree $i$. Thus, \eq{pt2eq8} is the tensor product of the symmetric algebras on $K_{2i}^\sht(X)\ot_\Z\Q,$ and the exterior algebras on $K_{2i-1}^\sht(X)\ot_\Z\Q,$ for $i\ge 1$.

\begin{dfn}
\label{pt2def5}
Let $X$ be a smooth, connected projective 3-fold over $\C$. There is a natural group morphism $\up:\Z\op A_1^\alg(X)\op\Z\ra K_0^\sht(X)$ such that if $E^\bu=[V\ot\O_X\,{\buildrel\rho\over\longra}\, F]$ is a perfect complex on $X$ with $V$ a finite-dimensional vector space over $\C$ and $F$ a 1-dimensional sheaf on $X$, such that $V\ot\O_X$ is in degree $-1$ and $F$ in degree 0, then $\lb E^\bu\rb=\up(\dim_\C V,\lb F\rb)$, where $\lb F\rb\in A_1^\alg(X)\op\Z$ is as in \S\ref{pt21}. We identify stable pairs $\O_X\,{\buildrel s\over\longra}\, F$ as objects $\cI^\bu$ in the derived category $D^b\coh(X)$, where $\O_X$ is in degree $-1$ and $F$ in degree $0$. For $(\be,n)\in A_1^\alg(X)\op\Z$, there is a universal complex $[\O_{X\t P_n(X,\be)}\,{\buildrel\s\over\longra}\,\fF]$ on $X\t P_n(X,\be)$, and this induces a morphism
\e
\ze_{\be,n}:P_n(X,\be)\ra \M_{\up(1,\be,n)}.
\label{pt2eq9}
\e

After tensoring the usual Pandharipande--Thomas virtual class on $P_n^\alg(X,\be)$ with $\Q$, we
write
\e
[P_n^\alg(X,\be)]_\virt
\in H_{2c_1(X)\cdot\be}(\M_{\up(1,\be,n)},\Q)
\label{pt2eq10}
\e
for its pushforward under $(\ze_{\be,n})_*$.

We will also use the image of \eq{pt2eq10} under $\Pi^\pl:\M_{\up(1,\be,n)}\ra\M^\pl_{\up(1,\be,n)}$:
\e
\Pi^\pl_*\bigl([P_n^\alg(X,\be)]_\virt\bigr)\in H_{2c_1(X)\cdot\be}(\M^\pl_{\up(1,\be,n)},\Q).
\label{pt2eq11}
\e

Now using \eq{pt2eq10} and the isomorphism \eq{pt2eq7}, we can instead write 
\e
[P_n^\alg(X,\be)]^0_\virt\in H_{2c_1(X)\cdot\be}(\M_0,\Q).
\label{pt2eq12}
\e
The superscript 0 in $[P_n^\alg(X,\be)]^0_\virt$ means we move it from $\M_{\up(1,\be,n)}$ to~$\M_0$. 
\end{dfn}

The reason we introduced $A_1^\alg(X)$ above is that the class $\lb F\rb\in H_2(X,\Z)\op\Z$ might not determine the class of $F$ in $K_0^\sht(X)$, so $\up$ might not be well-defined as a map $\Z\op H_2(X,\Z)\op\Z\ra K_0^\sht(X)$, but the definition of algebraic equivalence of 1-cycles ensures that the class $\lb F\rb\in A_1^\alg(X)\op\Z$ does determine~$[F]\in K_0^\sht(X)$.

Note that although $[P_n^\alg(X,\be)]_\virt$ is defined in $\Z$-homology, we project it to $\Q$-homology, as the proofs of Theorems \ref{pt1thm1} and \ref{pt1thm2} work over $\Q$, not~$\Z$.

The composition $\Pi^\pl\ci\ze_{\be,n}:P_n(X,\be)\ra \M^\pl_{\up(1,\be,n)}$ embeds $P_n^\alg(X,\be)$ as an open substack of $\M^\pl_{\up(1,\be,n)}$. In some ways \eq{pt2eq11} is more natural than \eq{pt2eq10}, and the wall-crossing formulae \eq{pt2eq40}--\eq{pt2eq42} below involve \eq{pt2eq11}. However, the cohomology classes $\tau_k(\eta)$ in \eq{pt2eq2} live in $H^*(\M_{\up(1,\be,n)},\Q)$ rather than $H^*(\M_{\up(1,\be,n)}^\pl,\Q)$, so to define the Pandharipande--Thomas invariants \eq{pt2eq3} we need to start from \eq{pt2eq10}, not \eq{pt2eq11}. See \S\ref{pt26} on how to relate \eq{pt2eq10} and~\eq{pt2eq11}.

\subsection{Graded vertex algebras and graded Lie algebras}
\label{pt23}

For background on vertex algebras, we recommend Frenkel--Ben-Zvi \cite{FrBZ}.

\begin{dfn}
\label{pt2def6}
Let $V_*=\bigop_{n\in\Z}V_n$ be a graded $\Q$-vector space. Form the vector space $V_*[[z]][z^{-1}]$ of $V_*$-valued Laurent series in a formal variable $z$, and make it $\Z$-graded by declaring $\deg z=-2$. A {\it field\/} on $V_*$ is a $\Q$-linear map $V_*\ra V_*[[z]][z^{-1}]$, graded of some degree. The set of all fields on $V_*$ is denoted $\cF(V_*)$ and is considered as a graded $\Q$-vector space by declaring $\cF(V_*)_n$ to be the set of degree $n$ fields $V_*\ra V_*[[z]][z^{-1}]$ for~$n\in\Z$.

A {\it graded vertex algebra\/} $(V_*,\boo,D,Y)$ over $\Q$ is a $\Z$-graded $\Q$-vector space $V_*$ with an identity element $\boo\in V_0$, a $\Q$-linear {\it translation operator\/} $D:V_*\ra V_{*+2}$ of degree 2, and a grading-preserving {\it state-field correspondence\/} $Y:V_*\ra\cF(V_*)_*$ written $Y(u,z)v=\sum_{n\in\Z}u_n(v)z^{-n-1}$, where $u_n$ maps $V_*\ra V_{*+a-2n-2}$ for $u\in V_a$, satisfying:
\begin{itemize}
\setlength{\itemsep}{0pt}
\setlength{\parsep}{0pt}
\item[(i)] $Y(\boo,z)v=v$ for all $v\in V_*$.
\item[(ii)] $Y(v,z)\boo=e^{zD}v$ for all $v\in V_*$.
\item[(iii)] For all $u\in V_a$ and $v\in V_b$, there exists $N\gg 0$ such that for all $w\in V_*$
\end{itemize}
\begin{equation*}
(z_1-z_2)^N\bigl(Y(u,z_1)Y(v,z_2)w-(-1)^{ab}Y(v,z_2)Y(u,z_1)w\bigr)=0\;\>\text{in $V_*[[z_1^{\pm 1},z_2^{\pm 1}]]$.}
\end{equation*}
\end{dfn}

\begin{dfn}
\label{pt2def7}
A {\it graded Lie algebra\/} over $\Q$ is a pair $(V_*,[\,,\,])$, where $V_*=\bigop_{a\in\Z}V_a$ is a graded $\Q$-vector space, and $[\,,\,]:V_*\t V_*\ra V_*$ is a $\Q$-bilinear map called the {\it Lie bracket}, which is graded (that is, $[\,,\,]$ maps $V_a\t V_b\ra V_{a+b}$ for all $a,b\in\Z$), such that for all $a,b,c\in\Z$ and $u\in V_a$, $v\in V_b$ and $w\in V_c$ we have:
\begin{gather*}
[v,u]=(-1)^{ab+1}[u,v],\\
(-1)^{ca}[[u,v],w]+(-1)^{ab}[[v,w],u]+(-1)^{bc}[[w,u],v]=0.
\end{gather*}
\end{dfn}

The next proposition is due to Borcherds \cite[\S 4]{Borc}.

\begin{prop}
\label{pt2prop1}
Let\/ $(V_*,\boo,D,Y)$ be a graded vertex algebra over $\Q$. We may construct a graded Lie algebra $(\check V_*,[\,,\,])$ over $\Q$ as follows. Noting the shift in grading, define a $\Z$-graded\/ $\Q$-vector space $\check V_*$ by 
\begin{equation*}
\check V_n=V_{n+2}/D(V_n)\qquad\text{for\/ $n\in\Z,$} 
\end{equation*}
so that\/ $\check V_*=V_{*+2}/D(V_*)$. If\/ $u\in V_{a+2}$ and\/ $v\in V_{b+2},$ the Lie bracket on $\check V_*$ is
\e
\bigl[u+D(V_a),v+D(V_b)\bigr]=u_0(v)+D(V_{a+b})\in\check V_{a+b}.
\label{pt2eq13}
\e
\end{prop}

\subsection{Vertex algebras on the homology of moduli stacks}
\label{pt24}

If $\A$ is a well-behaved $\C$-linear additive category, such as $\coh(X)$, $D^b\coh(X)$ for $X$ a smooth projective $\C$-scheme, or $\modCQ$, $D^b\modCQ$ for $Q$ a quiver, and $\M$ is the moduli stack of objects in $\A$, the second author \cite{Joyc5} defines a graded vertex algebra structure on the Betti $\Q$-homology $H_*(\M)$. These vertex algebras are important in the enumerative invariant theory of \cite{Joyc6}. We explain them when~$\A=D^b\coh(X)$.

\begin{dfn}
\label{pt2def8}
Let $X$ be a smooth, connected, projective $\C$-scheme with $\dim_\C X=m$, and use the notation of \S\ref{pt22}, with the moduli stack $\M=\coprod_{\al\in K_0^\sht(X)}\M_\al$ of objects in $D^b\coh(X)$ and its Betti homology $H_*(\M)$ over $\Q$ as in \eq{pt2eq6}. The {\it Euler form\/} is the biadditive map $\chi:K_0(\coh(X))\t K_0(\coh(X))\ab\ra\Z$ defined for all $E^\bu,F^\bu\in D^b\coh(X)$ by
\begin{equation*}
\chi([E^\bu],[F^\bu])=\sum_{i\in\Z}(-1)^i\dim_\C\Ext^i(E^\bu,F^\bu).
\end{equation*}
It factors via the projection $K_0(\coh(X))\twoheadrightarrow K_0^\sht(X)$. The Grothendieck--Riemann--Roch Theorem \cite[\S A.4]{Hart} says that
\begin{equation*}
\chi(\lb E^\bu\rb,\lb F^\bu\rb)=\int_X\ch(E^\bu)^\vee\cup\ch(F^\bu)\cup\td(X),
\end{equation*}
where $\ch:K_0^\sht(X)\ra H^{\rm even}(X,\Q)$ is the Chern character, $(E^\bu)^\vee$ is the derived dual, and $\td(X)$ the Todd class, as in~\cite[App.~A]{Hart}. 

There is a {\it universal perfect complex\/} $\cU^\bu\ra X\t\M$ such that $\cU^\bu\vert_{X\t\{[E^\bu]\}}\cong E^\bu$. The {\it Ext complex\/} $\cExt^\bu$ is a perfect complex on $\M\t\M$, given by
\begin{equation*}
\cExt^\bu=(\Pi_{23})_*\bigl[\Pi_{12}^*((\cU^\bu)^\vee)\ot\Pi_{13}^*(\cU^\bu)\bigr],
\end{equation*}
where $\Pi_{ij}$ projects to the product of the $i^{\rm th}$ and $j^{\rm th}$ factors of $X\t\M\t\M$. It has $H^i\bigl(\cExt^\bu\vert_{([E^\bu],[F^\bu])}\bigr)\cong\Ext^i(E^\bu,F^\bu)$ for $E^\bu,F^\bu\in D^b\coh(X)$ and $i\in\Z$. We write $\cExt^\bu_{\al,\be}$ for the restriction of $\cExt^\bu$ to $\M_\al\t\M_\be\subseteq\M\t\M$, for $\al,\be$ in $K_0^\sht(X)$. Then~$\rank\bigl(\cExt^\bu_{\al,\be}\bigr)=\chi(\al,\be)$.

There is a natural morphism of stacks $\Phi:\M\t\M\ra\M$ which on $\C$-points acts by $\Phi_*:([E^\bu],[F^\bu])\mapsto[E^\bu\op F^\bu]$, for all objects $E^\bu,F^\bu\in D^b\coh(X)$, and on isotropy groups acts by $\Phi_*:\Iso_{\M\t\M}([E^\bu],[F^\bu])\cong\Aut(E^\bu)\t\Aut(F^\bu)\ra \Iso_\M([E^\bu\op F^\bu])\cong\Aut(E^\bu\op F^\bu)$ by $(\la,\mu)\mapsto\bigl(\begin{smallmatrix}\la & 0 \\ 0 & \mu\end{smallmatrix}\bigr)$ for $\la\in\Aut(E^\bu)$ and $\mu\in\Aut(F^\bu)$, using the obvious matrix notation for $\Aut(E^\bu\op F^\bu)$. That is, $\Phi$ is the morphism of moduli stacks induced by direct sum in the additive category $D^b\coh(X)$. It is associative and commutative in~$\Ho(\HSta_\C)$.

There is a natural morphism of stacks $\Psi:[*/\bG_m]\t\M\ra\M$ which on $\C$-points acts by $\Psi_*:(*,[E^\bu])\mapsto[E^\bu]$, for all objects $E^\bu$ in $D^b\coh(X)$, and on isotropy groups acts by $\Psi_*:\Iso_{[*/\bG_m]\t\M}(*,[E^\bu])\cong\bG_m\t\Aut(E^\bu)\ra \Iso_\M([E^\bu])\cong\Aut(E^\bu)$ by $(\la,\mu)\mapsto \la\mu=(\la\cdot\id_{E^\bu})\ci\mu$ for $\la\in\bG_m$ and $\mu\in\Aut(E^\bu)$. We have identities in~$\Ho(\HSta_\C)$:
\begin{align*}
\Psi\ci(\id_{[*/\bG_m]}\t\Phi)&=\Phi\ci\bigl(\Psi\ci\Pi_{12},\Psi\ci\Pi_{13}\bigr):
[*/\bG_m]\t\M^2\longra\M,\\
\Psi\ci(\id_{[*/\bG_m]}\t\Psi)&=\Psi\ci(\Om\t\id_\M):
[*/\bG_m]^2\t\M\longra\M,
\end{align*}
where $\Pi_{ij}$ projects to the $i^{\rm th}$ and $j^{\rm th}$ factors, and
$\Om:[*/\bG_m]^2\ra[*/\bG_m]$ is induced by the morphism $\bG_m\t\bG_m\ra\bG_m$ mapping $(\la,\mu)\mapsto\la\mu$. Write
\begin{align*}
\Phi_{\al,\be}&=\Phi\vert_{\M_\al\t\M_\be}:\M_\al\t\M_\be\longra\M_{\al+\be},\\
\Psi_\al&=\Psi\vert_{[*/\bG_m]\t\M_\al}:[*/\bG_m]\t\M_\al\longra\M_\al.
\end{align*}

The quotient stack $[*/\bG_m]$ has topological realization $[*/\bG_m]^\top\simeq\CP^\iy$. Thus we may write
\e
H_*([*/\bG_m])\cong \Q[t] \quad\text{with $\deg t=2$, and $[\CP^n]=t^n$,}
\label{pt2eq14}
\e
where $\CP^n\hookra\CP^\iy$ is the standard inclusion.

We will define a graded vertex algebra structure on the homology $H_*(\M)$. The inclusion of the zero object $0\in D^b\coh(X)$ gives a morphism $[0]:*\hookra\M$ inducing $\Q\cong H_0(*)
\ra H_0(\M)$, and we define $\boo\in H_0(\M)$ to be the image of $1\in\Q$ under this. Define the translation operator $D:H_*(\M)\ra H_{*+2}(\M)$ by
\e
D(u)=\Psi_*(t\bt u)
\label{pt2eq15}
\e
where $t\in H_2([*/\bG_m])$ is as in \eq{pt2eq14}, and
\begin{equation*}
\bt:H_2([*/\bG_m])\t H_k(\M)\longra H_{k+2}([*/\bG_m]\t\M)
\end{equation*}
is the exterior tensor product in homology, and $\Psi_*:H_{k+2}([*/\bG_m]\t\M)\ra H_{k+2}(\M)$ is pushforward along $\Psi$.

Using \eq{pt2eq6}, for $u\in H_*(\M_\al)\subset H_*(\M)$ and $v\in H_*(\M_\be)\subset H_*(\M)$, define
\e
\begin{split}
&Y(u,z)v=Y(z)(u\ot v)= (-1)^{\chi(\al,\be)} \sum\nolimits_{i,j\ge 0} z^{\chi(\al, \be)+\chi(\be,\al)-i+j}\cdot{}\\
&\bigl(\Phi_{\al,\be}\ci(\Psi_\al\t\id_{\M_\be})\bigr)_*\bigl(t^j\bt ((u \bt v) \cap c_i((\cExt_{\al,\be}^\bu)^\vee\op \si_{\al,\be}^*(\cExt^\bu_{\be,\al})))\bigr),
\end{split}
\label{pt2eq16}
\e
where $t^j\in H_{2j}([*/\bG_m])$ is as in \eq{pt2eq14}, and $\si_{\al,\be}:\M_\al\t\M_\be\ra\M_\be\t\M_\al$ exchanges the factors. Using \eq{pt2eq6}, for $n\in\Z$ and $\al\in K_0^\sht(X)$ we write
\e
\hat H_n(\M_\al)= H_{n-2\chi(\al,\al)}(\M_\al),\quad \hat H_n(\M)=\bigop_{\al\in K_0^\sht(X)}\hat H_n(\M_\al).
\label{pt2eq17}
\e
That is, $\hat H_*(\M)$ is $H_*(\M)$, but with grading shifted by $-2\chi(\al,\al)$ in the component $H_*(\M_\al)\subset H_*(\M)$. The second author \cite{Joyc5} proves:	
\end{dfn}

\begin{thm}
\label{pt2thm1}
$(\hat H_*(\M), \boo,D,Y)$ above is a graded vertex algebra over $\Q$.
\end{thm}

\subsection{Lie algebras on the homology of moduli stacks}
\label{pt25}

Let $\A$ be a well-behaved $\C$-linear additive category, and $\M,\M^\pl$ be the usual moduli stack and the `projective linear' moduli stack of objects in $\A$, as in \S\ref{pt22}. We have seen in \S\ref{pt24} that $H_*(\M)$ is a graded vertex algebra. The second author \cite{Joyc5} shows $H_*(\M^\pl)$ is a {\it graded Lie algebra}, which is related to the vertex algebra structure on $H_*(\M)$ via Proposition \ref{pt2prop1}. These Lie algebras play a central r\^ole in the enumerative invariant theory of \cite{Joyc6}: we regard enumerative invariants as classes in $H_*(\M^\pl)$, and wall-crossing formulae are written using the Lie bracket on $H_*(\M^\pl)$. We explain them when~$\A=D^b\coh(X)$.

\begin{dfn}
\label{pt2def9}
Continue in the situation of Definition \ref{pt2def8}. As in Definition \ref{pt2def3} we have the projective linear moduli stack $\M^\pl=\coprod_{\al\in K_0^\sht(X)}\M_\al^\pl$ of objects in $D^b\coh(X)$. As for \eq{pt2eq17}, using \eq{pt2eq6}, for $n\in\Z$ and $\al\in K_0^\sht(X)$ we write
\e
\check H_n(\M_\al^\pl)\!=\!H_{n+2-2\chi(\al,\al)}(\M_\al^\pl),\;\> \check H_n(\M^\pl)\!=\bigop_{\!\!\!\!\!\!\!\al\in K_0^\sht(X)\!\!\!\!\!\!\!\!\!\!\!}\check H_n(\M_\al^\pl).
\label{pt2eq18}
\e
That is, $\check H_*(\M^\pl)$ is $H_*(\M^\pl)$, but with grading shifted by $2-2\chi(\al,\al)$ in the component $H_*(\M_\al^\pl)\subset H_*(\M^\pl)$. The next theorem is proved in~\cite{Joyc5}.
\end{dfn}

\begin{thm}
\label{pt2thm2}
Work in the situation of Definitions\/ {\rm\ref{pt2def8}} and\/ {\rm\ref{pt2def9},} and consider the graded Lie algebra $\hat H_{*+2}(\M)/D(\hat H_*(\M))$ constructed by combining Proposition\/ {\rm\ref{pt2prop1}} and Theorem\/ {\rm\ref{pt2thm1}}. Then $\Pi^\pl:\M\ra\M^\pl$ gives a morphism $\Pi^\pl_*:H_*(\M)\ab\ra H_*(\M^\pl),$ which maps $H_*(\M_\al)\ab\ra H_*(\M_\al^\pl)$ for $\al\in K_0^\sht(X)$. With the shifted gradings in \eq{pt2eq17} and\/ {\rm\eq{pt2eq18},} this maps $\hat H_{k+2}(\M)\ra\check H_k(\M^\pl)$ for $k\in\Z$. Then:
\begin{itemize}
\setlength{\itemsep}{0pt}
\setlength{\parsep}{0pt}
\item[{\bf(a)}] There is a graded Lie bracket\/ $[\,,\,]$ on $\check H_*(\M^\pl),$ constructed by Upmeier\/ {\rm\cite{Upme}} using `projective Euler classes', proving conjectures by the second author in the first version of\/ {\rm\cite{Joyc5}}. This Lie bracket has $\bigl[\check H_k(\M_\al^\pl),\check H_l(\M_\be^\pl)\bigr]\ab\subseteq \check H_{k+l}(\M_{\al+\be}^\pl)$ for $k,l\in\Z$ and $\al,\be\in K_0^\sht(X)$.
\item[{\bf(b)}] $D(\hat H_*(\M))$ lies in the kernel of\/ $\Pi^\pl_*:\hat H_{*+2}(\M)\ab\ra \check H_*(\M^\pl),$ so that\/ $\Pi^\pl_*$ descends to $\Pi^\pl_*:\hat H_{*+2}(\M)/D(\hat H_*(\M))\ab\ra\check H_*(\M^\pl)$. This is a morphism of graded Lie algebras.
\item[{\bf(c)}] If\/ $\ch\al\ne 0$ in $H^*(X,\Q)$ then\/ $\Pi^\pl_*:\hat H_{*+2}(\M_\al)/D(\hat H_*(\M_\al))\ab\ra\check H_*(\M_\al^\pl)$ is an isomorphism.
\end{itemize}
\end{thm}

\begin{rem}
\label{pt2rem1}
{\bf(a)} Theorem \ref{pt2thm2} shows that except in $\check H_*(\M_\al^\pl)\!\subset\!\check H_*(\M^\pl)$ with $\ch\al=0$, the graded Lie algebras $\hat H_{*+2}(\M)/D(\hat H_*(\M))$ and $\check H_*(\M^\pl)$ are isomorphic. In fact we never use $\check H_*(\M_\al^\pl)$ with $\ch\al=0$, as we are interested in enumerative invariants $[\M_\al^\ss(\tau)]_\inv\in H_*(\M_\al^\pl)$ which are only defined when $\ch\al\ne 0$. So we can work with $\hat H_{*+2}(\M)/D(\hat H_*(\M))$ rather than~$\check H_*(\M^\pl)$.
\smallskip

\noindent{\bf(b)} If we took $\M$ to be the moduli stack of objects in the abelian category $\coh(X)$ rather than the derived category $D^b\coh(X)$ then if $\ch\al=0$ we would have $\M_\al=\M_\al^\pl=\{[0]\}$ if $\al=0$ and $\M_\al=\M_\al^\pl=\es$ otherwise, and $\Pi^\pl_*:\hat H_{*+2}(\M_\al)/D(\hat H_*(\M_\al))\ab\ra\check H_*(\M_\al^\pl)$ would be an isomorphism for all $\al$ in $K_0^\sht(X)$, not just $\al$ with $\ch\al\ne 0$. The failure of the isomorphism when $\ch\al=0$ is a peculiarity of working with derived categories.
\smallskip

\noindent{\bf(c)} One might think that Lie algebras are simpler than vertex algebras. However, the Lie bracket on $\check H_*(\M^\pl)$ is a deep and complicated object. In practice, the easiest way to compute it is usually to lift to $\hat H_{*+2}(\M)$ and use vertex algebras. We discuss helpful techniques for doing this in~\S\ref{pt26}.
\end{rem}

\subsection{\texorpdfstring{Lifting calculations from $H_*(\M^\pl)$ to $H_*(\M)$}{Lifting calculations from H₊(ℳᵖˡ) to H₊(ℳ)}}
\label{pt26}

The next theorem, proved in \cite{Joyc5}, helps us to understand $H_*(\M^\pl,\Q)$ and its Lie bracket in terms of $H_*(\M,\Q)$ and its vertex algebra structure.

\begin{thm}
\label{pt2thm3}
Let\/ $X,m,K_0^\sht(X),$ $\M,\M_\al,\M^\pl,\M^\pl_\al,$ $\Pi^\pl:\M_\al\ra\M_\al^\pl,$ and\/ $H_*(\M_\al,\Q),H^*(\M_\al,\Q)$ be as in Definitions {\rm\ref{pt2def3}--\ref{pt2def4},} and\/ $\cU^\bu\ra X\t\M,$ $D:H_*(\M_\al,\Q)\ra H_{*+2}(\M_\al,\Q)$ be as in Definition\/~{\rm\ref{pt2def8}}.

Suppose $\al\in K_0^\sht(X)$ with\/ $\ch\al\ne 0$ in $H^{\rm even}(X,\Q),$ and pick\/ $i=0,\ldots,m$ and\/ $\eta\in H_{2i}(X,\Q)$ with\/ $\ch_i(\al)\cdot\eta\ne 0$ in $\Q$. Define $e_\eta=\ch_{i+1}(\cU^\bu)\backslash \eta$ in $H^2(\M_\al,\Q)$. Define $R_\al^\eta:H_*(\M_\al,\Q)\ra H_{*-2}(\M_\al,\Q)$ by $R_\al^\eta(u)=u\cap e_\eta$. Define $H_*(\M_\al,\Q)_{e_\eta=0}=\Ker(R_\al^\eta)\subset H_*(\M_\al,\Q)$. Then:
\smallskip

\noindent{\bf(a)} $[R_\al^\eta,D]=(\ch_i(\al)\cdot\eta)\,\id:H_*(\M_\al,\Q)\ra H_*(\M_\al,\Q)$.
\smallskip

\noindent{\bf(b)} For all\/ $k,l\ge 0$ with\/ $k\le l/2,$ $D^k:H_{l-2k}(\M_\al,\Q)_{e_\eta=0}\ra H_l(\M_\al,\Q)$ is injective, so that\/ $D^k\bigl(H_{l-2k}(\M_\al,\Q)_{e_\eta=0}\bigr)$ is a vector subspace of\/ $H_l(\M_\al,\Q)$ isomorphic to $H_{l-2k}(\M_\al,\Q)_{e_\eta=0}$. Furthermore
\e
H_l(\M_\al,\Q)=\bigop_{0\le k\le l/2}D^k\bigl(H_{l-2k}(\M_\al,\Q)_{e_\eta=0}\bigr).
\label{pt2eq19}
\e
Hence
\e
H_l(\M_\al,\Q)=H_l(\M_\al,\Q)_{e_\eta=0}\op D\bigl(H_{l-2}(\M_\al,\Q)\bigr).
\label{pt2eq20}
\e
Therefore, by Theorem\/ {\rm\ref{pt2thm2}(c),} we have an isomorphism
\e
\Pi^\pl_*\vert_{H_l(\M_\al,\Q)_{e_\eta=0}}:H_l(\M_\al,\Q)_{e_\eta=0}\,{\buildrel\cong\over\longra}\,H_l(\M^\pl_\al,\Q).
\label{pt2eq21}
\e
Write the inverse of\/ \eq{pt2eq21} as
\e
I_{e_\eta=0}:H_l(\M^\pl_\al,\Q)\,{\buildrel\cong\over\longra}\,H_l(\M_\al,\Q)_{e_\eta=0}.
\label{pt2eq22}
\e

\noindent{\bf(c)} Write $\Pi_{e_\eta=0}:H_l(\M_\al,\Q)\ra H_l(\M_\al,\Q)_{e_\eta=0}$ for the projection to the first factor in \eq{pt2eq20}. Then
\e
\Pi_{e_\eta=0}=\sum_{0\le k\le l/2}\frac{1}{k!(-\ch_i(\al)\cdot\eta)^k}D^k\ci (R_\al^\eta)^k.
\label{pt2eq23}
\e
Note that by combining {\bf(a)} and\/ \eq{pt2eq23} we can show that\/ $R_\al^\eta\ci\Pi_{e_\eta=0}=0,$ consistent with\/ $H_*(\M_\al,\Q)_{e_\eta=0}=\Ker(R_\al^\eta)$.
\smallskip

\noindent{\bf(d)} Suppose $\al,\be\in K_0^\sht(X)$ with\/ $\ch_i(\al)\cdot\eta,\ch_i(\be)\cdot\eta,\ch_i(\al+\be)\cdot\eta$ all nonzero, and let\/ $u\in H_a(\M_\al^\pl,\Q),$ $v\in H_b(\M_\be^\pl,\Q),$ so that\/ $[u,v]\in H_*(\M_{\al+\be}^\pl,\Q)$ using the Lie bracket from Theorem\/ {\rm\ref{pt2thm2}(a)}. Then in $H_*(\M_{\al+\be},\Q)$ we have
\ea
&I_{e_\eta=0}\bigl([u,v]\bigr)=
\label{pt2eq24}\\
&\sum_{0\le k\le (a+b)/2}\frac{1}{k!}\Bigl(\frac{-\ch_i(\al)\cdot\eta}{\ch_i(\al\!+\be)\cdot\eta}\Bigr)^kD^k\bigl[\bigl(I_{e_\eta=0}(u)\bigr)_k\bigl(I_{e_\eta=0}(v)\bigr)\bigr],
\nonumber
\ea
using the vertex algebra structure from Theorem\/ {\rm\ref{pt2thm1}}.
\smallskip

\noindent{\bf(e)} In {\bf(d)\rm,} suppose instead that\/ $\ch_i(\al)\cdot\eta\ne 0$ and $\ch_i(\be)\cdot\eta=0,$ and choose an arbitrary $v'\in H_b(\M_\be,\Q)$ with $\Pi^\pl_*(v')=v\in H_b(\M_\be^\pl,\Q)$. Then
\e
I_{e_\eta=0}\bigl([u,v]\bigr)=\sum_{0\le k\le (a+b)/2}\frac{(-1)^k}{k!}D^k\bigl[\bigl(I_{e_\eta=0}(u)\bigr)_k(v')\bigr].
\label{pt2eq25}
\e
\end{thm}

We can use this theorem to lift calculations in the Lie algebra $H_*(\M^\pl,\Q)$ to the vertex algebra $H_*(\M,\Q)$. For $\al\in K_0^\sht(X)$ with $\ch_i(\al)\cdot\eta\ne 0$ we identify $H_l(\M^\pl_\al,\Q)$ with $H_l(\M_\al,\Q)_{e_\eta=0}$ using \eq{pt2eq21}--\eq{pt2eq22}, and then we compute Lie brackets in $H_*(\M^\pl,\Q)$ in the vertex algebra using~\eq{pt2eq24}--\eq{pt2eq25}. 

A useful case in Theorem \ref{pt2thm3} is $i=0$ and $\eta=1\in H_0(X,\Q)$, and then $\ch_i(\al)\cdot\eta=\rank\al$. Recall that in Definition \ref{pt2def5} we defined $[P_n^\alg(X,\be)]_\virt\in H_{2c_1(X)\cdot\be}(\M_{\up(1,\be,n)},\Q)$ in \eq{pt2eq10} as the image of $\ze_{\be,n}$ in \eq{pt2eq9}, and considered its pushforward $\Pi^\pl_*\bigl([P_n^\alg(X,\be)]_\virt\bigr)\in H_{2c_1(X)\cdot\be}(\M^\pl_{\up(1,\be,n)},\Q)$ in \eq{pt2eq11}. Note that $\rank\up(1,\be,n)=-1\ne 0$. We claim these classes are related by
\e
[P_n^\alg(X,\be)]_\virt=I_{e_1=0}\ci\Pi^\pl_*\bigl([P_n^\alg(X,\be)]_\virt\bigr).
\label{pt2eq26}
\e

To see this, note that the universal complex $[\O_{X\t P_n(X,\be)}\,{\buildrel\s\over\longra}\,\fF]$ on $X\t P_n(X,\be)$ has $c_1(\cdots)=0$, since $c_1(\O_{X\t P_n(X,\be)})=0$, and $c_1(\fF)=0$ as the support of $\fF$ has codimension 2. Therefore
\begin{equation*}
R^1_{\up(1,\be,n)}\bigl([P_n^\alg(X,\be)]_\virt\bigr)\!=\![P_n^\alg(X,\be)]_\virt\cap\bigl(c_1\bigl([\O_{X\t P_n(X,\be)}\,{\buildrel\s\over\longra}\,\fF]\bigr)\backslash 1\bigr)\!=\!0,
\end{equation*}
so $[P_n^\alg(X,\be)]_\virt$ lies in $\Ker(R^1_{\up(1,\be,n)})=H_{2c_1(X)\cdot\be}(\M_{\up(1,\be,n)},\Q)_{e_1=0}$, and equation \eq{pt2eq26} follows as \eq{pt2eq21} and \eq{pt2eq22} are inverse.

\subsection{Equivariant (co)homology of stacks}
\label{pt27}

Let $G$ be a linear algebraic group over $\C$ and $S$ be a projective complex manifold, or a $\C$-scheme, or an Artin $\C$-stack, or a higher $\C$-stack, with a $G$-action. Then we can define the $G$-{\it equivariant cohomology groups\/} $H_G^k(S,\Q)$ for $k=0,1,2,\ldots,$ and the $G$-{\it equivariant homology groups\/} $H^G_k(S,\Q)$ for~$k\in\Z$.

Here we must work with a particular kind of equivariant (co)homology of stacks, explained in \cite{Joyc5} and \cite{Joyc6}. If $S$ is a $\C$-stack with an action of $G$ then the equivariant cohomology $H_G^*(S,\Q)$ is just $H^*([S/G],\Q)$, the ordinary cohomology of the quotient stack~$[S/G]$. 

However, equivariant homology $H^G_*(S,\Q)$ is more complicated. We define $H^G_k(S,\Q)$ for all $k\in\Z$, not just for $k\ge 0$. It is a module over $H_G^{-*}(S,\Q)$ via the cap product $\cap:H^G_k(S,\Q)\t H_G^l(S,\Q)\ra H^G_{k-l}(S,\Q)$. We write $H_G^{-*}(S,\Q)$ here rather than $H_G^*(S,\Q)$ to make clear that cohomology gradings are subtracted from homology gradings. Since the projection $\pi:S\ra *$ induces an algebra morphism $\pi^*:H_G^*(*,\Q)\ra H_G^*(S,\Q)$, $H^G_*(S,\Q)$ is also a module over $H_G^{-*}(*,\Q)$. If the $G$-action on $S$ is trivial then
\e
H^G_k(S,\Q)\cong\prod_{j\ge 0}H^j_G(*,\Q)\ot_\Q H_{j+k}(S,\Q).
\label{pt2eq27}
\e
Thus, $H^G_*(S,\Q)$ behaves like homology for $S$, but like cohomology for~$G$.

We will study equivariant (co)homology using spectral sequences. See McCleary \cite{McCl} for a good introduction to these. There is a first quadrant cohomology spectral sequence with $E_2$ page $E_2^{p,q}=H_G^p(*,\Q)\ot_\Q H^q(S,\Q)$ converging to $H_G^{p+q}(S,\Q)$, coming from the fibration $[S/G]\ra[*/G]$ with fibre $S$. The differentials on the $k^{\rm th}$ page act as $d_k^{p,q}:E_k^{p,q}\ra E_k^{p+k,q+1-k}$. The $(k+1)^{\rm st}$ page is obtained by taking cohomology on the $k^{\rm th}$ page by
\e
E_{k+1}^{p,q}=\frac{\Ker(d_k^{p,q}:E_k^{p,q}\longra E_k^{p+k,q+1-k})}{\Im(d_k^{p-k,q-1+k}:E_k^{p-k,q-1+k}\longra E_k^{p,q})}\,.
\label{pt2eq28}
\e
They converge to the $\iy$-page $E_\iy^{p,q}$ of the spectral sequence, that is, $E_\iy^{p,q}=E_k^{p,q}$ for $k\gg 0$. Then $H^k_G(S,\Q)$ has a filtration $H^k_G(S,\Q)=F^0H^k_G(S,\Q)\supseteq F^1H^k_G(S,\Q)\supseteq\cdots$, with $\bigcap_{p\ge 0}F^pH^k_G(S,\Q)=0$, and isomorphisms for~$p\ge 0$
\begin{equation*}
\frac{F^pH^k_G(S,\Q)}{F^{p+1}H^k_G(S,\Q)}\cong E_\iy^{p,k-p}.
\end{equation*}

Similarly, there is a second quadrant homology spectral sequence with $E^2$ page $E^2_{p,q}=H_G^{-p}(*,\Q)\ot_\Q H_q(S,\Q)$ converging to $H^G_{p+q}(S)$. The differentials on the $k^{\rm th}$ page act as $d^k_{p,q}:E^k_{p,q}\ra E^k_{p-k,q-1+k}$. The $(k+1)^{\rm st}$ page is obtained by taking cohomology on the $k^{\rm th}$ page by
\e
E^{k+1}_{p,q}=\frac{\Ker(d^k_{p,q}:E^k_{p,q}\longra E^k_{p-k,q-1+k})}{\Im(d^k_{p+k,q+1-k}:E^k_{p+k,q+1-k}\longra E^k_{p,q})}\,.
\label{pt2eq29}
\e
They converge to the $\iy$-page $E^\iy_{p,q}$ of the spectral sequence, that is, $E^\iy_{p,q}=E^k_{p,q}$ for $k\gg 0$. Then $H_k^G(S)$ has a filtration $H_k^G(S,\Q)=F^0H_k^G(S,\Q)\supseteq F^{-1}H_k^G(S,\Q)\supseteq\cdots$, with $\bigcap_{p\le 0}F^pH_k^G(S,\Q)=0$, and isomorphisms for~$p\le 0$
\e
\frac{F^pH_k^G(S,\Q)}{F^{p-1}H_k^G(S,\Q)}\cong E^\iy_{p,k-p}.
\label{pt2eq30}
\e

Since $E^\iy_{p,k-p}=0$ if $k-p<0$, we see that if $k<0$ then
\e
H_k^G(S,\Q)=F^0H_k^G(S,\Q)=F^{-1}H_k^G(S,\Q)=\cdots=F^kH_k^G(S,\Q).
\label{pt2eq31}
\e

Both spectral sequences have graded actions of $H^*_G(*,\Q)$ preserving all the structure. That is, there are $\Q$-bilinear multiplication maps $H^n_G(*,\Q)\t E_k^{p,q}\ra E_k^{p+n,q}$ and $H^n_G(*,\Q)\t E^k_{p,q}\ra E^k_{p-n,q}$ for all $k=2,3,\ldots,\iy$ and $n,p,q$, which when $k=2$ come from multiplication $H^n_G(*,\Q)\t H^{\pm p}(*,\Q)\ra H^{n\pm p}(*,\Q)$ and the identity on $H^q(S,\Q)$ or $H_q(S,\Q)$. The natural multiplication $H^n_G(*,\Q)\t H^k_G(S,\Q)\ra H^{k+n}_G(S,\Q)$ maps $F^pH^k_G(S,\Q)\ra F^{p+n}H^{k+n}_G(S,\Q)$, and the natural multiplication $H^n_G(*,\Q)\t H_k^G(S,\Q)\ra H_{k-n}^G(S,\Q)$ maps $F^pH_k^G(S,\Q)\ra F^{p-n}H_{k-n}^G(S,\Q)$. These multiplication maps commute with the $d_k^{p,q},d^k_{p,q}$ and are compatible with the isomorphisms~\eq{pt2eq28}--\eq{pt2eq30}.

Many natural operations on $H^*_G(S,\Q),H^G_*(S,\Q)$ are compatible with the filtrations $(F^pH^k_G(S,\Q))_{p\ge 0},(F^pH_k^G(S,\Q))_{p\le 0}$. For example, if $\Phi:S\ra T$ is a $G$-equivariant morphism then $\Phi^*:H^k_G(T,\Q)\ra H^k_G(S,\Q)$ maps $F^pH^k_G(T,\Q)\ra F^pH^k_G(S,\Q)$ and $\Phi_*:H_k^G(S,\Q)\!\ra\! H_k^G(T,\Q)$ maps $F^pH_k^G(S,\Q)\!\ra\! F^pH_k^G(T,\Q)$. The cap product $H_k^G(S,\Q)\t H^l_G(S,\Q)\ra H_{k-l}^G(S,\Q)$ maps $F^pH_k^G(S,\Q)\t F^{p'}H^l_G(S,\Q)\ra F^{p-p'}H_{k-l}^G(S,\Q)$.

For $N\ge 0$, define the ($N$-){\it truncated\/ $G$-equivariant homology group}
\begin{equation*}
H_k^{G,\le N}(S,\Q)=\frac{H_k^G(S,\Q)}{F^{-N-1}H_k^G(S,\Q)}\,.
\end{equation*}
Then \eq{pt2eq31} implies that $H_k^{G,\le N}(S,\Q)=0$ for $k<-N$. Operations on $H_*^G(S,\Q)$ such as pushforwards and cap product with classes in $H^*_G(S,\Q)$ descend to $H_*^{G,\le N}(S,\Q)$. The definition of equivariant homology in \cite{Joyc5,Joyc6} implies that $(F^pH_k^G(S,\Q))_{p\le 0}$ is a {\it complete\/} filtration, and $H^G_*(S,\Q)$ is the inverse limit $\varprojlim_{N\ra\iy}H^{G,\le N}_*(S,\Q)$. In the action of $H^*_G(*,\Q)$ on $H_k^{G,\le N}(S,\Q)$, multiplication by $H^{>N}_G(*,\Q)$ gives zero. Thus, we can roughly think of $H_*^{G,\le N}(S,\Q)$ as being the result of killing the action of $H^{>N}_G(*,\Q)$ on $H_*^G(S,\Q)$. For example, if the $G$-action on $S$ is trivial then modifying \eq{pt2eq27}, we have
\begin{equation*}
H^{G,\le N}_k(S,\Q)\cong\bigop_{j=0}^NH^j_G(*,\Q)\ot_\Q H_{j+k}(S,\Q).
\end{equation*}

\subsection[Extension of \S\ref{pt21}--\S\ref{pt22} and \S\ref{pt24}--\S\ref{pt26} to equivariant (co)homology]{Extension of \S\ref{pt21}--\S\ref{pt22} and \S\ref{pt24}--\S\ref{pt26} to equivariant \\ (co)homology}
\label{pt28}

Let $X$ be a smooth, connected projective 3-fold over $\C$, and suppose a linear algebraic $\C$-group $G$ acts on $X$, and acts trivially on $A_1^\alg(X)$ and $H_2(X,\Z)$. For example, $X$ could be toric, and $G=\bG_m^3$. Then the material of \S\ref{pt21} on Pandharipande--Thomas classes $[P_n(X,\be)]_\virt$ and invariants $\bigl(\ts\prod_{i=1}^m\tau_{k_i}(\eta_i)\bigr)\cdot[P_n(X,\be)]_\virt$, and of \S\ref{pt22} on (co)homology of moduli stacks $\M,\M^\pl,\M_\al,\M_\al^\pl$, and of \S\ref{pt24}--\S\ref{pt26} on vertex algebra and Lie algebra structures on homology $H_*(\M,\Q)$, $H_*(\M^\pl,\Q)$, can all be promoted to ($N$-truncated) $G$-equivariant (co)homology, as in \S\ref{pt27}. We make some comments on this.
\smallskip

\noindent{\bf(a)} $G$ acts on the moduli spaces $P_n(X,\be)$ preserving the obstruction theories. Thus virtual classes are defined in $G$-equivariant homology
\begin{equation*}
[P_n(X,\be)]^G_\virt\in H_{2c_1(X)\cdot\be}^G(P_n(X,\be),\Z).
\end{equation*}
For each $N\ge 0$ there is also an $N$-truncated version
\begin{equation*}
[P_n(X,\be)]^{G,\le N}_\virt\in H_{2c_1(X)\cdot\be}^{G,\le N}(P_n(X,\be),\Z).
\end{equation*}

\noindent{\bf(b)} To define $G$-equivariant descendent insertions we take $\eta\in H^l_G(X,\Q)$ rather than $H^l(X,\Q)$, and then we define $\tau_k(\eta)\in H^{2k+l-2}_G(P_n(X,\be),\Q)$ as in~\eq{pt2eq2}.
\smallskip

\noindent{\bf(c)} For products of descendent insertions $\prod_{i=1}^m\tau_{k_i}(\eta_i)$, instead of requiring that $\sum_{i=1}^m(2k_i+l_i-2)=2c_1(X)\cdot\be$, we require only that $\sum_{i=1}^m(2k_i+l_i-2)\ge 2c_1(X)\cdot\be$. Then we can define $G$-{\it equivariant Pandharipande--Thomas invariants\/}
\begin{gather*}
PT^G_{\be,n}\bigl(\ts\prod_{i=1}^m\tau_{k_i}(\eta_i)\bigr)=\bigl(\ts\prod_{i=1}^m\tau_{k_i}(\eta_i)\bigr)\cdot[P_n(X,\be)]^G_\virt\\
\text{in $H^{\sum_{i=1}^m(2k_i+l_i-2)-2c_1(X)\cdot\be}_G(*,\Q)$,}
\end{gather*}
generalizing \eq{pt2eq3}. As for \eq{pt2eq4} we combine these into a generating function
\begin{equation*}
PT^G_\be\bigl(\ts\prod\limits_{i=1}^m\tau_{k_i}(\eta_i),q\bigr)=\sum\limits_{n\in\Z}PT^G_{\be,n}\bigl(\ts\prod\limits_{i=1}^m\tau_{k_i}(\eta_i)\bigr)q^n \quad\text{in $H^*_G(*,\Q)[[q]][q^{-1}]$.}
\end{equation*}
Equivariant Pandharipande--Thomas invariants are discussed in~\cite{Pand,PaTh1,PaTh2,PaPi1,PaPi2}.
\smallskip

\noindent{\bf(d)} $G$ acts on the moduli stacks $\M,\M_\pl$ in \S\ref{pt22}, and preserves the substacks $\M_\al,\M^\pl_\al$ (at least for $\al$ of the form $\up(r,\be,n)$). The morphism $\ze_{\be,n}$ in \eq{pt2eq9} is $G$-equivariant, so as in \eq{pt2eq10} we have an equivariant virtual class
\begin{equation*}
[P_n^\alg(X,\be)]^G_\virt\in H_{2c_1(X)\cdot\be}^G(\M_{\up(1,\be,n)},\Q).
\end{equation*}

\noindent{\bf(e)} In the situation of \S\ref{pt24}--\S\ref{pt26}, suppose $G$ is a linear algebraic group over $\C$ which acts on the smooth projective $m$-fold $X$. Then $G$ also acts on $D^b\coh(X)$ and on the moduli stacks $\M,\M^\pl$. In \cite{Joyc5} the second author generalizes the vertex algebra and Lie algebra structures on $H_*(\M),H_*(\M^\pl)$ in \S\ref{pt24}--\S\ref{pt25} to {\it equivariant\/} homology $H_*^G(\M),H_*^G(\M^\pl)$, and $H_*^{G,\le N}(\M),H_*^{G,\le N}(\M^\pl)$, as in~\S\ref{pt27}.

For each $N\ge 0$, Definition \ref{pt2def8} for the vertex algebra structure on $H_*(\M,\Q)$ extends without change to $H_*^{G,\le N}(\M,\Q)$, using truncated equivariant homology $H_*^{G,\le N}(\cdots)$ and equivariant cohomology $H^*_G(\cdots)$ throughout.

To apply the same construction to the full equivariant homology $H_*^G(\M,\Q)$, there is a subtlety. Consider the definition \eq{pt2eq16} of the vertex algebra operation $Y(u,z)v$ in $H_*^G(\M,\Q)$. Suppose $u\in H_a^G(\M_\al,\Q)$ and $v\in H_b^G(\M_\be,\Q)$. Then the coefficient of $z^k$ in $Y(u,z)v$ is
\ea
&[z^k]\bigl\{Y(u,z)v\bigr\}=
\label{pt2eq32}\\
&\sum_{i\ge\max(0,\chi(\al,\be)+\chi(\be,\al)-k)\!\!\!\!\!\!\!\!\!\!\!\!\!\!\!\!\!\!\!} \begin{aligned}[t]
&(-1)^{\chi(\al,\be)}\bigl(\Phi_{\al,\be}\ci(\Psi_\al\t\id_{\M_\be})\bigr)_*\bigl(t^{i+k-\chi(\al,\be)-\chi(\be,\al)} 
\bt{} \\
&\qquad\qquad\quad ((u \bt v) \cap c_i((\cExt_{\al,\be}^\bu)^\vee\op\si_{\al,\be}^*(\cExt^\bu_{\be,\al})))\bigr).
\end{aligned}
\nonumber
\ea
Here the term $(u \bt v)\cap c_i(\cdots)$ in the final line lies in $H^G_{a+b-2i}(\M_\al\t\M_\be,\Q)$. In non-equivariant homology this is zero if $i>(a+b)/2$, so that \eq{pt2eq32} is a finite sum. In truncated equivariant homology $H^{G,\le N}_{a+b-2i}(\M_\al\t\M_\be,\Q)$ it is zero if $i>(a+b+N)/2$, so again \eq{pt2eq32} is a finite sum. But for $H_*^G(\M,\Q)$ we may have $H^G_n(\M_\al\t\M_\be,\Q)\ne 0$ as $n\ra -\iy$, so \eq{pt2eq32} can be an infinite sum, which converges in the filtered sense using the complete filtration $(F^pH^G_n(\M_\al\t\M_\be,\Q))_{p\le 0}$ explained in~\S\ref{pt27}.

Thus, the vertex algebra operations in Definition \ref{pt2def8} are well-defined for $H_*^G(\M,\Q)$, with \eq{pt2eq16} a convergent infinite sum. However, $H_*^G(\M,\Q)$ may not be a vertex algebra in the strict sense of Definition \ref{pt2def6}, as $Y(u,z)v$ may not lie in $H_*^G(\M,\Q)[[z]][z^{-1}]$, but instead lies in the filtered completion of this, allowing powers of $z$ unbounded below; so $H_*^G(\M,\Q)$ is a `filtered graded vertex algebra'. But $H_*^{G,\le N}(\M,\Q)$ is a true graded vertex~algebra. 

Both $H_*^G(\M^\pl,\Q)$ and $H_*^{G,\le N}(\M^\pl,\Q)$ are honest graded Lie algebras, though the Lie bracket in $H_*^G(\M^\pl,\Q)$ may be a convergent infinite sum.

For the equivariant generalization of \S\ref{pt26}, in the truncated case $H_*^{G,\le N}(\cdots)$ we must replace the sum over $0\le k\le l/2$ in \eq{pt2eq19} by $0\le k\le (l+N)/2$, since $H_{l-2k}^{G,\le N}(\M_\al,\Q)=0$ if $k>(l+N)/2$. Similarly, for $H_*^{G,\le N}(\cdots)$, the upper limits on $k$ in the sums in \eq{pt2eq23}--\eq{pt2eq25} are increased by $N/2$. For $H_*^G(\cdots)$, the sums over $k$ in \eq{pt2eq19} and \eq{pt2eq23}--\eq{pt2eq25} should be over all $k\ge 0$, and they are infinite sums which converge in the filtered sense.

\subsection{Stability conditions and combinatorial coefficients}
\label{pt29}

The next definition comes from the second author \cite{Joyc3}. See also Rudakov~\cite{Ruda}.

\begin{dfn}
\label{pt2def10}
Let $\A$ be an abelian category, and $K_0(\A)$ its Grothendieck group. Suppose we are given a surjective quotient $K_0(\A)\twoheadrightarrow K(\A)$. We write $\lb E\rb\in K(\A)$ for the class of $E\in\A$. Suppose $0\in\A$ is the only object in class $0\in K(\A)$. The {\it positive cone\/} $C(\A)\subset K(\A)\sm\{0\}$ is~$C(\A)=\bigl\{\lb E\rb:0\ne E\in\A\bigr\}$.

Let $(T,\leq)$ be a totally ordered set and $\tau:C(\A)\ra T$ be a map. We call $(\tau,T,\leq)$ a {\it weak stability condition\/} on $\A$ if for all $\al,\be,\ga \in C(\A)$ with $\be=\al+\ga$, either 
$\tau(\al) \leq \tau(\be) \leq \tau(\ga)$, or~$\tau(\al) \geq \tau(\be) \geq \tau(\ga)$.

We call $(\tau,T,\leq)$ a {\it stability condition\/} if for all such $\al,\be,\ga$, either 
$\tau(\al)<\tau(\be)<\tau(\ga)$, or $\tau(\al)>\tau(\be)>\tau(\ga)$, or~$\tau(\al)=\tau(\be)=\tau(\ga)$.

Let $(\tau,T,\leq)$ be a weak stability condition. An object $E$ of $\A$ is called:
\begin{itemize}
\setlength{\itemsep}{0pt}
\setlength{\parsep}{0pt}
\item[(i)] {\it $\tau$-stable\/} if $\tau([E'])<\tau([E/E'])$ for all subobjects $E'\subset E$ with $E'\ne 0,E$.
\item[(ii)] {\it $\tau$-semistable\/} if $\tau([E'])\!\leq\!\tau([E/E'])$ for all $E'\subset E$ with $E'\ne 0,E$.
\item[(iii)] {\it $\tau$-unstable} if it is not $\tau$-semistable.
\end{itemize}
\end{dfn}

We now define universal combinatorial coefficients $S,U,\ti U(\al_1,\ldots,\al_n;\tau,\ti\tau)$ which appear in wall-crossing formulae for enumerative invariants under change of stability condition. They first appeared in the second author's series \cite{Joyc1,Joyc2,Joyc3,Joyc4} on motivic invariants counting $\tau$-(semi)stable objects in abelian categories. They were then applied to Donaldson--Thomas theory of Calabi--Yau 3-folds in Joyce--Song \cite{JoSo}, and to invariants in homology in Gross--Joyce--Tanaka \cite{GJT} and the second author \cite{Joyc6}. The next definition comes from \cite[\S 4.1]{Joyc4}, but with notation changed as in~\cite[\S 3.3]{JoSo}.

\begin{dfn}
\label{pt2def11}
Let $\A$ be an abelian category, and choose $K_0(\A)\twoheadrightarrow K(\A)$ as in Definition \ref{pt2def10}. Let
$(\tau,T,\le),(\ti\tau,\ti T,\le)$ be weak stability conditions on
$\A$.

Let $n\ge 1$ and $\al_1,\ldots,\al_n\in C(\A)$. If for all
$i=1,\ldots,n-1$ we have either
\begin{itemize}
\setlength{\itemsep}{0pt}
\setlength{\parsep}{0pt}
\item[(a)] $\tau(\al_i)\le\tau(\al_{i+1})$ and
$\ti\tau(\al_1+\cdots+\al_i)>\ti\tau(\al_{i+1}+\cdots+\al_n)$, or
\item[(b)] $\tau(\al_i)>\tau(\al_{i+1})$ and~$\ti\tau(\al_1+\cdots+\al_i)\le\ti\tau(\al_{i+1}+\cdots+\al_n)$,
\end{itemize}
then define $S(\al_1,\ldots,\al_n;\tau,\ti\tau)=(-1)^r$, where $r$ is the number of $i=1,\ldots,n-1$ satisfying (a). Otherwise define $S(\al_1,\ldots,\al_n;\tau,\ti\tau)=0$. Now
define
\begin{align*}
&U(\al_1,\ldots,\al_n;\tau,\ti\tau)=\\
&\sum_{\begin{subarray}{l} \phantom{wiggle}\\
1\le l\le m\le n,\;\> 0=a_0<a_1<\cdots<a_m=n,\;\>
0=b_0<b_1<\cdots<b_l=m:\\
\text{Define $\be_1,\ldots,\be_m\in C(\A)$ by
$\be_i=\al_{a_{i-1}+1}+\cdots+\al_{a_i}$.}\\
\text{Define $\ga_1,\ldots,\ga_l\in C(\A)$ by
$\ga_i=\be_{b_{i-1}+1}+\cdots+\be_{b_i}$.}\\
\text{We require $\tau(\be_i)=\tau(\al_j)$, $i=1,\ldots,m$,
$a_{i-1}<j\le a_i$,}\\
\text{and $\ti\tau(\ga_i)=\ti\tau(\al_1+\cdots+\al_n)$,
$i=1,\ldots,l$}
\end{subarray}
\!\!\!\!\!\!\!\!\!\!\!\!\!\!\!\!\!\!\!\!\!\!\!\!\!\!\!\!\!\!\!\!\!
\!\!\!\!\!\!\!\!\!\!\!\!\!\!\!\!\!\!\!\!\!\!\!\!\!\!\!\!\!\!\!\!\!
\!\!\!\!\!\!\!\!\!\!\!\!\!\!\!\!\!\!\!\!}
\begin{aligned}[t]
\frac{(-1)^{l-1}}{l}\cdot\prod\nolimits_{i=1}^lS(\be_{b_{i-1}+1},
\be_{b_{i-1}+2},\ldots,\be_{b_i}; \tau,\ti\tau)&\\
\cdot\prod_{i=1}^m\frac{1}{(a_i-a_{i-1})!}&\,.
\end{aligned}
\end{align*}
\end{dfn}

The next theorem is proved in \cite[Th.~5.4]{Joyc4} (see also \cite[Th.~3.14]{JoSo}). It describes a property of the coefficients $U(-;\tau,\ti\tau)$, it does not matter what $\cL$ and $\ep^\al(\tau),\ep^\al(\ti\tau)$ are. We have no explicit definition for $\ti U(\al_1,\ldots,\al_n;\tau,\ti\tau)$, we only show that \eq{pt2eq33} can be rewritten in the form~\eq{pt2eq34}.

\begin{thm}
\label{pt2thm4}
Work in the situation of Definition\/ {\rm\ref{pt2def11}}. Let\/ $\cL$ be a Lie algebra over $\Q,$ and write $U(\cL)$ for its universal enveloping algebra, with product $*$. Suppose we are given elements $\ep^\al(\tau),\ep^\al(\ti\tau)\in\cL$ for $\al\in C(\A)$ satisfying
\e
\begin{gathered}
\ep^\al(\ti\tau)= \!\!\!\!\!\!\!
\sum_{\begin{subarray}{l}n\ge 1,\;\al_1,\ldots,\al_n\in
C(\A):\\ \al_1+\cdots+\al_n=\al\end{subarray}} \!\!\!\!\!\!\!
\begin{aligned}[t]
U(\al_1,&\ldots,\al_n;\tau,\ti\tau)\cdot{}\\
&\ep^{\al_1}(\tau)*\ep^{\al_2}(\tau)*\cdots* \ep^{\al_n}(\tau)
\end{aligned}
\end{gathered}
\label{pt2eq33}
\e
for each\/ $\al\in C(\A),$ with only finitely many nonzero terms. Then\/ \eq{pt2eq33} may be rewritten as an equation in the Lie algebra $\cL$ using the Lie bracket\/ $[\,,\,]$. That is, we may rewrite \eq{pt2eq33} in the form
\e
\begin{gathered}
\ep{}^\al(\ti\tau)= \!\!\!\!\!\!\!
\sum_{\begin{subarray}{l}n\ge 1,\;\al_1,\ldots,\al_n\in
C(\A):\\ \al_1+\cdots+\al_n=\al\end{subarray}} \!\!\!\!\!\!\!
\begin{aligned}[t]
\ti U(\al_1,&\ldots,\al_n;\tau,\ti\tau)\,\cdot\\
&[[\cdots[[\ep{}^{\al_1}(\tau),\ep{}^{\al_2}(\tau)],\ep{}^{\al_3}(\tau)],
\ldots],\ep{}^{\al_n}(\tau)],
\end{aligned}
\end{gathered}
\label{pt2eq34}
\e
for some system of combinatorial coefficients\/ $\ti U(\al_1,\ldots,\al_n;\tau,\ti\tau)\in\Q,$ with only finitely many nonzero terms, such that if we expand\/ $[f,g]=f*g-g*f$ then \eq{pt2eq34} becomes \eq{pt2eq33}.
\end{thm}

\subsection{One-dimensional Donaldson--Thomas invariants}
\label{pt210}

{\it Donaldson--Thomas invariants\/} count moduli stacks $\M_\al^\ss(\tau)$ of $\tau$-semistable coherent sheaves $F$ with $\lb F\rb\!=\!\al\!\in\! K_0^\sht(X)$ on a smooth (quasi)projective 3-fold $X$. They are only defined under extra assumptions on $X,\al$. One problem in doing this is that the natural obstruction theory $\phi:\cE^\bu\ra\bL_{\M_\al^\ss(\tau)}$ on $\M_\al^\ss(\tau)$ is perfect in $[-2,1]$. To define a Behrend--Fantechi virtual class \cite{BeFa} we need it to be perfect in $[-1,1]$, and this happens only if $\Ext^3(F,F)=0$ for all $[F]\in\M_\al^\ss(\tau)$.

There are actually (at least) {\it three different kinds\/} of Donaldson--Thomas invariants, which deal with the $\Ext^3(F,F)$ terms in different ways:
\begin{itemize}
\setlength{\itemsep}{0pt}
\setlength{\parsep}{0pt}
\item[(i)] If $X$ is a Calabi--Yau 3-fold then $\Ext^3(F,F)\cong\Hom(F,F)^*$ by Serre duality. If $\M_\al^\rst(\tau)=\M_\al^\ss(\tau)$ this gives canonical isomorphisms $\Ext^3(F,F)\cong\C$ for all $[F]\in\M_\al^\ss(\tau)$, and the $\Ext^3(F,F)$ terms can be deleted from the obstruction theory. See Thomas \cite{Thom}, Joyce--Song \cite{JoSo}, and Kontsevich--Soibelman \cite{KoSo} for more details.
\item[(ii)] If $X$ is any smooth projective 3-fold and we consider moduli stacks of rank 1 torsion-free sheaves $F$ with fixed determinant $\det F=\O_X$, we can define the obstruction theory using trace-free Ext groups $\Ext^i(F,F)_0$, and $\Ext^3(F,F)_0=0$ in this case. See Maulik--Nekrasov--Okounkov--Pandharipande \cite{MNOP1,MNOP2} for more.
\item[(iii)] If $X$ is a Fano 3-fold and $\dim\al>0$, or more generally if $X,\al$ satisfy some positivity condition involving $c_1(X)$, then $\Ext^3(F,F)=0$ for all $[F]$ in $\M_\al^\ss(\tau)$, and we can define virtual classes using the natural obstruction theory. See Thomas \cite{Thom} and the second author \cite{Joyc6} for more.
\end{itemize}

We should think of these as {\it different theories}, not one theory. The dimensions of the virtual classes in (i)--(iii) are different, and for (i),(iii) the wall-crossing formulae under change of stability condition in \cite{Joyc6,JoSo} are different.

We now explain Donaldson--Thomas invariants $\bigl[\M^\ss_{(\be,n)}(\mu^\om)\bigr]_\inv$ counting {\it one-dimensional\/} $\mu^\om$-semistable coherent sheaves $F$ on a smooth projective 3-fold $X$ with $\lb F\rb=(\be,n)$, for $\be$ a {\it superpositive\/} effective curve class.  Here $\be$ superpositive is the Fano-type condition we need to define invariants of type (iii) in this case. This is part of the theory of Donaldson--Thomas invariants of type (iii) and wall-crossing formulae developed by the second author in \cite{Joyc6}, proving conjectures in Gross--Joyce--Tanaka~\cite{GJT}.

\begin{dfn}
\label{pt2def12}
Let $X$ be a smooth, connected projective 3-fold over $\C$. Write $\coh_{\le 1}(X)\subset\coh(X)$ for the abelian subcategory of coherent sheaves $F\in\coh(X)$ of dimension $\le 1$, that is, $\dim_\C\supp F\le 1$. Define an abelian group $K(\coh_{\le 1}(X))=A_1^\alg(X)\op\Z$, for $A_1^\alg(X)$ as in Definition \ref{pt1def1}. There is a surjective group morphism $K_0(\coh_{\le 1}(X))\twoheadrightarrow K(\coh_{\le 1}(X))$ mapping $[F]\mapsto\lb F\rb$ for $F\in\coh_{\le 1}(X)$, where $\lb F\rb\in A_1^\alg(X)\op\Z$ is the class of $F$ from Definition \ref{pt1def2}. If $F$ is a nonzero pure 1-dimensional sheaf with $\lb F\rb=(\be,n)$ then $\be$ is an effective curve class. If $(\be,n)\in K(\coh_{\le 1}(X))$ then $(\be,n)\in C(\coh_{\le 1}(X))$ if and only if either $\be$ is an effective curve class and $n\in\Z$, or $\be=0$ and~$n>0$.

As for $\up$ in Definition \ref{pt2def5}, there is a natural morphism $\pi:K(\coh_{\le 1}(X))\ra K_0^\sht(X)$ mapping $\lb F\rb\mapsto\lb F\rb$ for $F\in\coh_{\le 1}(X)\subset\coh(X)$. We take curve classes in $A_1^\alg(X)$ rather than $H_2(X,\Z)$ to make $\pi$ well-defined.

Suppose $\om\in H^2(X,\R)$ is the K\"ahler class on $X$. Then $\om\cdot \Pi_\alg^\hom(\be)>0$ for every effective curve class $\be\in A_1^\alg(X)$. Define the {\it slope function\/}
\e
\mu^\om:C(\coh_{\le 1}(X))\ra\R\cup\{\iy_+\}\;\text{by}\; 
\mu^\om(\be,n)=\begin{cases}\displaystyle\frac{n}{\om\cdot \Pi_\alg^\hom(\be)}, & \be\ne 0, \\
\iy_+, & \be=0,
\end{cases}
\label{pt2eq35}
\e
where $\R\cup\{\iy_+\}$ has the obvious total order with $\R<\iy_+$. Then $\mu^\om$ is a stability condition on the abelian category $\A=\coh_{\le 1}(X)$ in the sense of Definition~\ref{pt2def10}.

Since $\coh_{\le 1}(X)\subset\coh(X)\subset D^b\coh(X)$, the moduli stack $\M_{\coh_{\le 1}(X)}$ of objects in $\coh_{\le 1}(X)$ is an open substack of the moduli stack $\M$ of objects in $D^b\coh(X)$ from Definition \ref{pt2def8}, and similarly $\M_{\coh_{\le 1}(X)}^\pl\subset\M^\pl$ is open.

As in \cite{Joyc6}, for all $(\be,n)\in C(\coh_{\le 1}(X))$ there are finite type open substacks $\M_{(\be,n)}^\rst(\mu^\om)\subseteq\M_{(\be,n)}^\ss(\mu^\om)\subset\M^\pl_{\pi(\be,n)}$ parametrizing $\mu^\om$-(semi)stable sheaves $F$ in $\coh_{\le 1}(X)$ with $\lb F\rb=(\be,n)$, where $\M_{(\be,n)}^\ss(\mu^\om)$ has a proper good moduli space. If $\M_{(\be,n)}^\rst(\mu^\om)=\M_{(\be,n)}^\ss(\mu^\om)$ then $\M_{(\be,n)}^\ss(\mu^\om)$ is a proper algebraic space, considered as a stack.

Suppose now that $\be$ is a {\it superpositive\/} effective curve class, as in Definition \ref{pt1def1}. Then it is shown in \cite{Joyc6} that $\Ext^3(F,F)=0$ for all $[F]\in\M_{(\be,n)}^\ss(\mu^\om)$. Because of this, by a standard construction \cite{Joyc6} we can define a perfect obstruction theory $\phi:\cE^\bu\ra\bL_{\M_{(\be,n)}^\ss(\mu^\om)}$ on $\M_{(\be,n)}^\ss(\mu^\om)$, with $H^i(\cE^\bu\vert_{[F]})=\Ext^{1-i}(F,F)^*$ for $i\ge 1$. As $\chi((\be,n),(\be,n))=0$ this has $\rank\cE^\bu=1$. The condition $\Ext^3(F,F)=0$ ensures that $\cE^\bu$ is perfect in $[-1,1]$ rather than~$[-2,1]$. 

If $\M_{(\be,n)}^\rst(\mu^\om)=\M_{(\be,n)}^\ss(\mu^\om)$ then $\M_{(\be,n)}^\ss(\mu^\om)$ is a proper algebraic space with a perfect obstruction theory, and so has a virtual class $[\M_{(\be,n)}^\ss(\mu^\om)]_\virt$ in homology by Behrend--Fantechi \cite{BeFa}. We regard this as lying in
\e
[\M_{(\be,n)}^\ss(\mu^\om)]_\virt\in H_2(\M_{\pi(\be,n)}^\pl,\Q)=\check H_0(\M_{\pi(\be,n)}^\pl,\Q),
\label{pt2eq36}
\e 
where the dimension is 2 as $\rank\cE^\bu=1$, and in the shifted grading \eq{pt2eq18} the degree is zero as $\chi(\pi(\be,n),\pi(\be,n))=0$.
\end{dfn}

The next theorem is proved in \cite{Joyc6}.

\begin{thm}
\label{pt2thm5}
{\bf(a)} In the situation of Definition\/ {\rm\ref{pt2def12},} if\/ $n\in\Z$ and\/ $\be\in A_1^\alg(X)$ is a \begin{bfseries}superpositive\end{bfseries} effective curve class then we can define a \begin{bfseries}dimension one Donaldson--Thomas invariant\end{bfseries}
\e
\bigl[\M^\ss_{(\be,n)}(\mu^\om)\bigr]_\inv\in H_2\bigl(\M_{\pi(\be,n)}^\pl,\Q\bigr)=\check H_0(\M_{\pi(\be,n)}^\pl,\Q).
\label{pt2eq37}
\e
If\/ $\M^\rst_{(\be,n)}(\mu^\om)=\M^\ss_{(\be,n)}(\mu^\om)$ then $[\M^\ss_{(\be,n)}(\mu^\om)]_\inv=[\M^\ss_{(\be,n)}(\mu^\om)]_\virt$ in\/ \eq{pt2eq36}. If\/ $\M^\rst_{(\be,n)}(\mu^\om)\ne\M^\ss_{(\be,n)}(\mu^\om)$ then $[\M^\ss_{(\be,n)}(\mu^\om)]_\inv$ has a complicated definition in {\rm\cite[\S 5.3]{Joyc6},} involving auxiliary pair invariants. If\/ $\ti\om$ is another K\"ahler class and\/ $\M^\ss_{(\be,n)}(\mu^{\ti\om})=\M^\ss_{(\be,n)}(\mu^\om)$ then $[\M^\ss_{(\be,n)}(\mu^{\ti\om})]_\inv=[\M^\ss_{(\be,n)}(\mu^\om)]_\inv$.
\smallskip

\noindent{\bf(b)} Suppose $\ti\om\in H^2(X,\R)$ is another K\"ahler class. Then
\ea
&[\M_{(\be,n)}^\ss(\mu^{\ti\om})]_\inv=
\label{pt2eq38}\\
&\sum_{\begin{subarray}{l}k\ge 1,\;(\be_1,n_1),\ldots,(\be_k,n_k)\in
C(\coh_{\le 1}(X)):\\ 
\text{$\be_i$ superpositive,} \\
(\be_1,n_1)+\cdots+(\be_k,n_k)=(\be,n)
\end{subarray}}\begin{aligned}[t]
&\ti U((\be_1,n_1),\ldots,(\be_k,n_k);\mu^\om,\mu^{\ti\om})\cdot{}\\
&\bigl[\bigl[\cdots\bigl[[\M_{(\be_1,n_1)}^\ss(\mu^\om)]_\inv,\\
&
[\M_{(\be_2,n_2)}^\ss(\mu^\om)]_\inv\bigr],\ldots\bigr],[\M_{(\be_k,n_k)}^\ss(\mu^\om)]_\inv\bigr]
\end{aligned}
\nonumber
\ea
in the Lie algebra $\check H_0(\M^\pl)$ from Theorem\/ {\rm\ref{pt2thm2}(a)}. Here\/ $\ti U(-;\mu^\om,\mu^{\ti\om})$ is as in Theorem\/ {\rm\ref{pt2thm4},} and there are only finitely many nonzero terms in \eq{pt2eq38}.
\smallskip

\noindent{\bf(c)} Suppose a linear algebraic\/ $\C$-group\/ $G$ acts on\/ $X,$ trivially on $A_1^\alg(X)$. Then the analogues of\/ {\bf(a)\rm,\bf(b)} hold in $G$-equivariant homology $H^G_2(\M_{\pi(\be,n)}^\pl,\Q),$ and in $H^{G,\le N}_2(\M_{\pi(\be,n)}^\pl,\Q)$ for all\/ $N\ge 0$. We write the corresponding invariants as $[\M^\ss_{(\be,n)}(\mu^\om)]^G_\inv$ and\/~$[\M^\ss_{(\be,n)}(\mu^\om)]^{G,\le N}_\inv$.
\end{thm}

An illustration of the difference between the Calabi--Yau and Fano cases (i),(iii) above is that Calabi--Yau Donaldson--Thomas invariants have dimension 0, so they lie in $\Z$ or $\Q$, but the Fano-type curve-counting invariants in Theorem \ref{pt2thm5} have real dimension 2, and lie in $H_2(\M_{(\be,n)}^\pl,\Q)$.

If we wish to lift $[\M^\ss_{(\be,n)}(\mu^\om)]_\inv$ from $H_2(\M_{\pi(\be,n)}^\pl,\Q)$ to $H_2(\M_{\pi(\be,n)},\Q)$ using the ideas of \S\ref{pt26}, we take $\om'\in H^2(X,\Q)$ to be a rational K\"ahler class on $X$ and choose $\eta=\PD(\om')\in H_4(X,\Q)$. If $\be$ is a superpositive curve class then $\ch_2(\pi(\be,n))\cdot\eta=\om'\cdot\be>0$, so Theorem \ref{pt2thm3} defines
\e
I_{e_\eta=0}\bigl([\M^\ss_{(\be,n)}(\mu^\om)]_\inv\bigr)\in H_2(\M_{\pi(\be,n)},\Q)_{e_\eta=0}\subset H_2(\M_{\pi(\be,n)},\Q).
\label{pt2eq39}
\e

\subsection{\texorpdfstring{An identity relating Pandharipande--Thomas and \\ one-dimensional Donaldson--Thomas invariants}{An identity relating Pandharipande--Thomas and one-dimensional Donaldson--Thomas invariants}}
\label{pt211}

Finally we give identities \eq{pt2eq40}--\eq{pt2eq42} relating Pandharipande--Thomas and one-dimensional Donaldson--Thomas invariants,  proved in \cite{Joyc6}, which we will use in \S\ref{pt3} to prove Theorem \ref{pt1thm2}. The method is similar to that used by Bridgeland \cite{Brid} and Toda \cite{Toda3,Toda4} to prove Conjecture \ref{pt1conj1} for Calabi--Yau 3-folds. The exact relation will be explained in Remark~\ref{pt2rem2}.

\begin{dfn}
\label{pt2def13}
Let $X$ be a smooth, connected projective 3-fold over $\C$. Define an abelian category $\acA$ to have {\it objects\/} $(F,V,\rho)$ where $F\in\coh_{\le 1}(X)$, $V$ is a finite-dimensional $\C$-vector space, and $\rho:V\ot_\C\O_X\ra F$ is a morphism in $\coh(X)$. If $(F,V,\rho),(F',V',\rho')$ are objects in $\acA$, a {\it morphism\/} $(\th,\phi):(F,V,\rho)\ra(F',V',\rho')$ consists of a morphism $\th:F\ra F'$ in $\coh(X)$ and a $\C$-linear map $\phi:V\ra V'$ such that the following commutes in $\coh(X)$:
\begin{equation*}
\xymatrix@C=100pt@R=15pt{ *+[r]{V\ot_\C\O_X} \ar[r]_{\phi\ot\id_{\O_X}} \ar[d]^\rho & *+[l]{V'\ot_\C\O_X} \ar[d]_{\rho'} \\ *+[r]{F} \ar[r]^\th & *+[l]{F'.\!} }	
\end{equation*}
Define composition of morphisms, and identities $(\id_F,\id_V)$, in the obvious way.

Define an abelian group $K(\acA)=\Z\op A_1^\alg(X)\op\Z$. Elements of $K(\acA)$ will be written $(d,\be,n)$. If $(F,V,\rho)\in\acA$, the {\it class\/} of $(F,V,\rho)$ is $\lb F,V,\rho\rb=(d,\be,n)$ where $d=\dim_\C V$ and $\lb F\rb=(\be,n)$ is as in Definition \ref{pt1def2}. This extends to a surjective group morphism $K_0(\acA)\twoheadrightarrow K(\acA)$ mapping $[F,V,\rho]\mapsto\lb F,V,\rho\rb$.

Define a functor $\ac\Pi:\acA\ra D^b\coh(X)$ to map $(F,V,\rho)$ to the complex $V\ot_\C\O_X\,{\buildrel\rho\over\longra}\,F$ with $V\ot_\C\O_X$ in degree $-1$ and $F$ in degree 0. This is a full and faithful functor which embeds $\acA$ as an abelian subcategory of $D^b\coh(X)$. The morphism $\up:\Z\op A_1^\alg(X)\op\Z\ra K_0^\sht(X)$ in Definition \ref{pt2def5} was defined to satisfy $\up(\lb F,V,\rho\rb)=\lb\ac\Pi(F,V,\rho)\rb$ for all~$(F,V,\rho)\in\acA$.

Let $\om\in H^2(X,\R)$ be a K\"ahler class on $X$. Let $\R\cup\{\iy,\iy_+\}$ have the obvious total order with $\R<\iy<\iy_+$. Fix $c\in\R\cup\{\iy\}$. Define
\begin{equation*}
\ac\mu^\om_c:C(\acA)\longra\R\cup\{\iy,\iy_+\}\quad\text{by}\quad 
\ac\mu^\om_c(d,\be,n)=\begin{cases}\mu^\om(\be,n), & d=0, \\
c, & d>0.
\end{cases}
\end{equation*}
Then $\ac\mu^\om_c$ is a weak stability condition on $\acA$ in the sense of \S\ref{pt29}. Note that we have two `infinities' $\iy<\iy_+$, where $\ac\mu^\om_c(0,0,n)=\iy_+$ for $n>0$ is the slope of a dimension 0 sheaf, and $\ac\mu^\om_\iy(d,\be,n)=\iy$ for~$d>0$.

In \cite{Joyc6} the second author shows that for $(d,\be,n)\in C(\acA)$ with $d=0,1$ there are finite type open moduli substacks $\M^\rst_{(d,\be,n)}(\ac\mu^\om_c)\subseteq\M^\ss_{(d,\be,n)}(\ac\mu^\om_c)\subseteq \M^\pl_{\up(d,\be,n)}$ parametrizing $\ac\mu^\om_c$-(semi)stable objects $(F,V,\rho)\in\acA$ with $\lb F,V,\rho\rb=(d,\be,n)$.

Take $\be$ to be an effective curve class and $d=1$. Then for $[F,V,\rho]\in \M^\ss_{(1,\be,n)}(\ac\mu^\om_c)$ we may take $V=\C$, so that $\rho:\O_X\ra F$. We find that:
\begin{itemize}
\setlength{\itemsep}{0pt}
\setlength{\parsep}{0pt}
\item[(i)] If $c<\mu^\om(\be,n)$ then $\M^\ss_{(1,\be,n)}(\ac\mu^\om_c)=\es$, since any $(F,\C,\rho)$ in class $(1,\be,n)$ is $\ac\mu^\om_c$-destabilized by the subobject $(F,0,0)\subset(F,\C,\rho)$.
\item[(ii)] $(F,\C,\rho)$ in class $(1,\be,n)$ is $\ac\mu^\om_{\iy}$-semistable if and only if $(F,\rho)$ is a Pand\-hari\-pan\-de--Thomas stable pair, in the sense of Definition \ref{pt1def2}. Thus we have an isomorphism~$\M^\ss_{(1,\be,n)}(\ac\mu^\om_\iy)\cong P^\alg_n(X,\be)$.
\end{itemize}
\end{dfn}

The next theorem is proved in \cite{Joyc6}.

\begin{thm}
\label{pt2thm6}
Work in the situation of Definition\/ {\rm\ref{pt2def13},} and let\/ $\be$ be a superpositive effective curve class on $X$. Then there exists\/ $C_\be\in\R$ such that:

\smallskip

\noindent{\bf(a)} If\/ $n\in\Z$ with\/ $\mu^\om(\be,n)>C_\be,$ and we choose small\/ $\ep>0$ in $\Q,$ then
\ea
\label{pt2eq40}\\[-17pt]
&0=\sum_{\begin{subarray}{l}
1\le j\le k, \\
\be=\be_1+\cdots+\be_k,  \\
n=n_1+\cdots+n_k,\\ 
\text{$\be_i$ effective and} \\
\text{$n_i\in\Z,$ $i\ne j,$} \\
\text{either $\be_j$ effective}\\
\text{and $n_j\in\Z,$}\\
\text{or $(\be_j,n_j)=(0,0)$}
\end{subarray}}\,\,
\begin{aligned}[t]&
\ti U\bigl((0,\be_1,n_1),\ldots,(0,\be_{j-1},n_{j-1}),(1,\be_j,n_j),\\
&(0,\be_{j+1},n_{j+1}),\ldots,(0,\be_k,n_k);\ac\mu^\om_\iy,\ac\mu^\om_{\mu^\om(\be,n)-\ep}\bigr)\cdot \\
&\bigl[\bigl[\cdots\bigl[[\M^\ss_{(\be_1,n_1)}(\mu^\om)]_\inv,[\M^\ss_{(\be_2,n_2)}(\mu^\om)]_\inv\bigr],\ldots,\\
&[\M^\ss_{(\be_{j-1},n_{j-1})}(\mu^\om)]_\inv\bigr],\Pi^\pl_*([P^\alg_{n_j}(X,\be_j)]_\virt)\bigr],\\
&[\M^\ss_{(\be_{j+1},n_{j+1})}(\mu^\om)]_\inv\bigr],\ldots,[\M^\ss_{(\be_k,n_k)}(\mu^\om)]_\inv\bigr].
\end{aligned}
\nonumber
\ea
Here\/ $\Pi^\pl_*\bigl([P^\alg_n(X,\be)]_\virt\bigr),[\M^\ss_{(\be_i,n_i)}(\mu^\om)]_\inv$ are as in {\rm\eq{pt2eq11}, \eq{pt2eq37},} and\/ $\ti U(\cdots)$ is as in Theorem\/ {\rm\ref{pt2thm4},} and the Lie brackets are in the Lie algebra $\check H_{\rm even}(\M^\pl)$ in Theorem\/ {\rm\ref{pt2thm2}(a)}. There are only finitely many nonzero terms in\/~\eq{pt2eq40}.
\smallskip

\noindent{\bf(b)} Since $\ti U\bigl((1,\be,n);\ac\mu^\om_\iy,\ac\mu^\om_{\mu^\om(\be,n)-\ep}\bigr)=1,$ by dividing \eq{pt2eq40} into terms with\/ $k=1$ and\/ $k\ge 2$ we see that provided\/ $\mu^\om(\be,n)>C_\be$ we have
\ea
&\Pi^\pl_*\bigl([P^\alg_n(X,\be)]_\virt\bigr)=
\label{pt2eq41}\\
&-\sum_{\begin{subarray}{l}
1\le j\le k,\; k\ge 2, \\
\be=\be_1+\cdots+\be_k,  \\
n=n_1+\cdots+n_k,\\ 
\text{$\be_i$ effective and} \\
\text{$n_i\in\Z,$ $i\ne j,$} \\
\text{either $\be_j$ effective}\\
\text{and $n_j\in\Z,$}\\
\text{or $(\be_j,n_j)=(0,0)$}
\end{subarray}}\,\,
\begin{aligned}[t]&
\ti U\bigl((0,\be_1,n_1),\ldots,(0,\be_{j-1},n_{j-1}),(1,\be_j,n_j),\\
&(0,\be_{j+1},n_{j+1}),\ldots,(0,\be_k,n_k);\ac\mu^\om_\iy,\ac\mu^\om_{\mu^\om(\be,n)-\ep}\bigr)\cdot \\
&\bigl[\bigl[\cdots\bigl[[\M^\ss_{(\be_1,n_1)}(\mu^\om)]_\inv,[\M^\ss_{(\be_2,n_2)}(\mu^\om)]_\inv\bigr],\ldots,\\
&[\M^\ss_{(\be_{j-1},n_{j-1})}(\mu^\om)]_\inv\bigr],\Pi^\pl_*([P^\alg_{n_j}(X,\be_j)]_\virt)\bigr],\\
&[\M^\ss_{(\be_{j+1},n_{j+1})}(\mu^\om)]_\inv\bigr],\ldots,[\M^\ss_{(\be_k,n_k)}(\mu^\om)]_\inv\bigr].
\end{aligned}
\nonumber
\ea
This writes Pandharipande--Thomas classes\/ $\Pi^\pl_*\bigl([P^\alg_n(X,\be)]_\virt\bigr)$ in terms of one-dim\-en\-sional Donaldson--Thomas invariants \sloppy $[\M^\ss_{(\be_i,n_i)}(\mu^\om)]_\inv$ and `lower down' Pan\-dhari\-pande--Thomas classes $\Pi^\pl_*\bigl([P^\alg_{n_j}(X,\be_j)]_\virt\bigr)$.

Note that although\/ \eq{pt2eq41} makes sense for arbitrary $(\be,n)\in C(\coh_{\le 1}(X))$ with\/ $\be$ superpositive, we prove it only when\/ $\mu^\om(\be,n)>C_\be,$ that is, when\/ $n>C_\be\,\om\cdot \Pi_\alg^\hom(\be)$. In general equation\/ \eq{pt2eq41} will \begin{bfseries}not\end{bfseries} hold for all\/~$n\in\Z$.
\smallskip

\noindent{\bf(c)} If\/ $\be$ is an \begin{bfseries}irreducible\end{bfseries} curve class then \eq{pt2eq41} reduces to
\e
\Pi^\pl_*\bigl([P^\alg_n(X,\be)]_\virt\bigr)=\bigl[[\M^\ss_{(\be,n)}(\mu^\om)]_\inv,\Pi^\pl_*([P^\alg_0(X,0)]_\virt)\bigr]
\label{pt2eq42}
\e
where $\Pi^\pl_*\bigl([P^\alg_0(X,0)]_\virt\bigr)\!\in\!H_0(\M_{\up(1,0,0)}^\pl)\!\cong\!\Q$ is the class of the point\/~$[0,\C,0]$.

\smallskip

\noindent{\bf(d)} Suppose a linear algebraic\/ $\C$-group\/ $G$ acts on\/ $X,$ trivially on $H_2(X,\Z)$. Then the analogues of\/ {\bf(a)\rm--\bf(c)} hold in $G$-equivariant homology $H^G_*(\M^\pl,\Q),$ and in truncated\/ $G$-equivariant homology $H^{G,\le N}_*(\M^\pl,\Q)$ for all\/~$N\ge 0$.\end{thm}

\begin{rem}
\label{pt2rem2}
{\bf(a)} Equation \eq{pt2eq40} is obtained by a wall-crossing formula in $\acA$ similar to \eq{pt2eq38}, changing from $\ac\mu^\om_\iy$-stability on the right hand side to $\ac\mu^\om_{\mu^\om(\be,n)-\ep}$-stability on the left hand side. As $\mu^\om(\be,n)-\ep<\mu^\om(\be,n)$ we have $\M^\ss_{(1,\be,n)}(\ac\mu^\om_{\mu^\om(\be,n)-\ep})=\es$, as in Definition \ref{pt2def13}(i), which is why the left hand side is zero. On the right hand side we have $\M^\ss_{(1,\be_j,n_j)}(\ac\mu^\om_\iy)=P^\alg_{n_j}(X,\be_j),$ as in Definition~\ref{pt2def13}(ii).
 
\smallskip

\noindent{\bf(b)} There is a projective linear moduli stack $\M^\pl_\acA$ of objects in $\acA$, and the full and faithful functor $\ac\Pi:\acA\hookra D^b\coh(X)$ induces a morphism of stacks $\ac\Pi_*:\M_\acA^\pl\hookra\M^\pl$. However, $\ac\Pi_*$ {\it is not an open inclusion on all of\/} $\M^\pl_\acA$. For example, if $F\in\coh_{\le 1}(X)$ and $\Ext^2(\O_X,F)\cong H^1(F\ot K_X)^*\ne 0$ then the object $(F,\C,0)$ has more deformations in $D^b\coh(X)$ than it does in $\acA$, so $T_{[F,\C,0]}\Pi_*:T_{[F,\C,0]}\M_\acA^\pl\ra T_{[\ac\Pi(F,\C,0)]}\M^\pl$ is injective but not surjective.

As the obstruction theory on $P^\alg_n(X,\be)$ used to define Pandharipande--Tho\-mas invariants is the natural one on $\M^\pl$, in Theorem \ref{pt2thm5} we define our invariants using the pullback of the obstruction theory on $\M^\pl$ to $\M_\acA^\pl$. But this is only valid in the open substack of $\M_\acA^\pl$ where $\ac\Pi_*$ is an open inclusion. We must restrict the classes $(1,\be,n)\in C(\acA)$ allowed to ensure that $\ac\Pi_*$ is an open inclusion on $\M^\ss_{(1,\be,n)}(\ac\mu^\om_c)$ for all $c\in[\mu^\om(\be,n)-\ep,\iy]$, so that the $\M^\ss_{(1,\be,n)}(\ac\mu^\om_c)$ have well-behaved obstruction theories. This is why we require $\mu^\om(\be,n)>C_\be$ in Theorem \ref{pt2thm5}. We use this to ensure that all $\ac\mu^\om$-semistable sheaves $F\in\coh_{\le 1}(X)$ involved in the wall-crossing have $\mu^\om(F)$ large enough to force~$H^1(F\ot K_X)=0$.

\smallskip

\noindent{\bf(c)} We could try to approach the problem another way, following Toda \cite[\S 5]{Toda4}, using invariants in the abelian subcategory $\A=\an{\O_X[1],\coh_{\le 1}(X)}$ in $D^b\coh(X)$, which contains $\ac\Pi(\acA)$, and using Toda's `limit stability conditions' $\tau^\th$ on $\A$. One of Toda's weak stability conditions $\tau^{1/2}$ is preserved by Verdier duality, so the $\tau^{1/2}$-invariants have a $\Z_2$-symmetry, which should be used to prove Conjecture \ref{pt1conj1}(c),(f), following Toda in the Calabi--Yau 3-fold case.

We do not do this because of technical limitations in the set up of \cite{Joyc6}. In \cite{Joyc6}, the invariants are defined, and the wall-crossing formulae proved, using auxiliary `pair invariants' and `quiver invariants' defined using `framing functors' on the abelian category $\A$. The second author can define these framing functors for $\acA$, but does not know how to do this for~$\A=\an{\O_X[1],\coh_{\le 1}(X)}$.

It seems likely that in future the theory of \cite{Joyc6} will be done without framing functors, and indeed Karpov--Moreira \cite{KaMo1} do not use framing functors in their K-theoretic version.
\end{rem}

\section{Proof of Theorem \ref{pt1thm2}}
\label{pt3}

\subsection{\texorpdfstring{Dependence of $[\M^\ss_{(\be,n)}(\mu^\om)]_\inv$ on $n$}{Dependence of [ℳˢˢ₍ₐ,ₙ₎(μʷ)]ᵢₙᵥ on n}}
\label{pt31}

Work in the situation of Theorem \ref{pt1thm2}. Fix an ample line bundle $L\ra X$, and take $\om=c_1(L)\in H^2(X,\Q)\subset H^2(X,\R)$ for the K\"ahler class used to define $\mu^\om$-stability in Definition \ref{pt2def12}. Let $\be\in A_1^\alg(X)$ be a superpositive effective curve class. Then Definition \ref{pt2def12} defines moduli stacks $\M_{(\be,n)}^\ss(\mu^\om)\subset\M_{\pi(\be,n)}^\pl$ for $n\in\Z$, and Theorem \ref{pt2thm5} defines invariants $[\M^\ss_{(\be,n)}(\mu^\om)]_\inv\in H_2(\M_{\pi(\be,n)}^\pl,\Q)$. 

Use Theorem \ref{pt2thm3} with $i=2$ and $\eta=\PD(\om)\in H_4(X,\Q)$ to get $I_{e_\eta=0}$ in \eq{pt2eq22}, and apply this to $[\M^\ss_{(\be,n)}(\mu^\om)]_\inv$ to get $I_{e_\eta=0}\bigl([\M^\ss_{(\be,n)}(\mu^\om)]_\inv\bigr)$ in $H_2(\M_{\pi(\be,n)},\Q)$ as in \eq{pt2eq39}. Using \eq{pt2eq7} we can instead take these to lie in
\e
I_{e_\eta=0}\bigl([\M^\ss_{(\be,n)}(\mu^\om)]_\inv\bigr){}^0\in H_2(\M_0,\Q),
\label{pt3eq1}
\e
where the superscript 0 means that we have transferred it from $H_*(\M_{(\be,n)},\Q)$ to $H_*(\M_0,\Q)$ using \eq{pt2eq7}. We can now ask how $I_{e_\eta=0}([\M^\ss_{(\be,n)}(\mu^\om)]_\inv)$ behaves as a function of $n\in\Z$, in the fixed vector space~$H_2(\M_0,\Q)$.

Mapping $E\mapsto L\ot E$ and $E^\bu\mapsto L\ot E^\bu$ gives equivalences of categories $\jmath^L:\coh(X)\ra\coh(X)$ and $\jmath^L:D^b\coh(X)\ra D^b\coh(X)$, and an isomorphism
\begin{equation*}
\jmath^L_*:K_0(\coh(X))\longra K_0(\coh(X)).
\end{equation*}
Also $\jmath^L$ induces isomorphisms of stacks $J^L:\M\ra\M$ and $J^{\pl,L}:\M^\pl\ra\M^\pl$, which yield isomorphisms on homology
\begin{equation*}
J^L_*:H_*(\M)\longra H_*(\M),\quad 
J^{\pl,L}_*:H_*(\M^\pl)\longra H_*(\M^\pl).
\end{equation*}
In the $G$-equivariant case in Theorem \ref{pt1thm2}(b) we also require that $L\ra X$ is a $G$-equivariant line bundle (which is possible by Brion \cite[Cor.~2.14]{Brio}), and then $J^L,J^{\pl,L}$ are $G$-equivariant.

If $F\in\coh_{\le 1}(X)$ with $\lb F\rb=(\be,n)$, one can show that $\chi(L\ot F)=\chi(F)+c_1(L)\cdot\Pi_\alg^\hom(\be)$. As $\om=c_1(L)$, it follows from \eq{pt2eq35} that $\mu^\om([L\ot F])=\mu^\om([F])+1$. Thus, although mapping $F\mapsto L\ot F$ changes $\mu^\om([F])$, it does not change the inequalities $\mu^\om([F'])\le \mu^\om([F/F'])$ which define when $F$ is $\mu^\om$-semistable, so $F$ is $\mu^\om$-semistable if and only if $L\ot F$ is. Thus $J^{\pl,L}$ maps
\begin{equation*}
J^{\pl,L}:\M^\ss_{(\be,n)}(\mu^\om)\,{\buildrel\cong\over\longra}\,\M^\ss_{(\be,n+d_\be)}(\mu^\om),
\end{equation*}
where $d_\be=c_1(L)\cdot\Pi_\alg^\hom(\be)$.

The graded vertex algebra on $H_*(\M)$ in \S\ref{pt24}, and the graded Lie algebra on $H_*(\M^\pl)$ in \S\ref{pt25}, and the construction of one-dimensional Donaldson--Thomas invariants in \S\ref{pt210}, are all invariant under $J^L,J^{\pl,L}$. Thus we see that
\e
J^{\pl,L}_*:[\M^\ss_{(\be,n)}(\mu^\om)]_\inv\longmapsto[\M^\ss_{(\be,n+d_\be)}(\mu^\om)]_\inv.
\label{pt3eq2}
\e

As $F\mapsto L\ot F$ does not change $\ch_2(F)\cdot \eta$ for $F\in\coh_{\le 1}(X)$, we find that $I_{e_\eta=0}\ci J^{\pl,L}_*=J^L_*\ci I_{e_\eta=0}$. Hence \eq{pt3eq2} implies that
\e
J^L_*:I_{e_\eta=0}\bigl([\M^\ss_{(\be,n)}(\mu^\om)]_\inv\bigr)\longmapsto I_{e_\eta=0}\bigl([\M^\ss_{(\be,n+d_\be)}(\mu^\om)]_\inv\bigr).
\label{pt3eq3}
\e
This also holds when we take the invariants to lie in $H_2(\M_0,\Q)$ as in \eq{pt3eq1}, since the $J^L_*$-action is compatible with~$H_2(\M_{\pi(\be,n)},\Q)\cong H_2(\M_0,\Q)$.

We use this to prove a polynomial property of the~$I_{e_\eta=0}([\M^\ss_{(\be,n)}(\mu^\om)]_\inv)^0$.

\begin{prop}
\label{pt3prop1}
{\bf(a)} Let\/ $\be\in A_1^\alg(X)$ be superpositive. Then there exist polynomials $P_j(n)\in H_2(\M_0,\Q)[n]$ for $1\le j\le d_\be,$ of degree $\le 6,$ with
\begin{equation*}
I_{e_\eta=0}\bigl([\M^\ss_{(\be,n)}(\mu^\om)]_\inv\bigr){}^0=P_j(n)\quad\text{if\/ $n\in\Z$ with\/ $n\equiv j\mod d_\be,$}
\end{equation*}
taking $I_{e_\eta=0}([\M^\ss_{(\be,n)}(\mu^\om)]_\inv)^0$ to lie in $H_2(\M_0,\Q)$ as in\/~\eq{pt3eq1}.
\smallskip

\noindent{\bf(b)} Suppose a linear algebraic $\C$-group $G$ acts on $X,$ trivially on $A_1^\alg(X),$ and\/ $N\ge 0$. Then the analogue of\/ {\bf(a)} holds for the $I_{e_\eta=0}([\M^\ss_{(\be,n)}(\mu^\om)]^{G,\le N}_\inv)^0$ in $H_2^{G,\le N}(\M_0,\Q),$ but with\/~$\deg P_j\le (N+1)(3(2+N)+1)-1$.
\end{prop}

\begin{proof}
For (a), using the notation of Definition \ref{pt2def3}, write $K_i^\sht(X)_\Q=K_i^\sht(X)\ab\ot_\Z\Q$ for $i\ge 0$. Then $K_0^\sht(X)_\Q$ is a $\Q$-algebra with product $\lb E^\bu\rb\cdot\lb F^\bu\rb=\lb E^\bu\ot^LF^\bu\rb$ and identity $\lb\O_X\rb$, with a representation on $K_i^\sht(X)_\Q$ for each $i\ge 1$. By Friedlander--Walker \cite[Th.~1.4]{FrWa2} there is a Chern character map $\ch:K_0^\sht(X)_\Q\ra A^*(X,\Q)=\bigop_{i=0}^3A^i(X,\Q)$, where $A^i(X,\Q)$ is the $\Q$-vector space of algebraic $(3-i)$-cycles on $X$ modulo algebraic equivalence, and $A^*(X,\Q)$ is an algebra under intersection. Furthermore $\ch$ is an isomorphism of $\Q$-algebras. We have $A^0(X,\Q)=\Q$ as $X$ is connected, and $\ch_0=\rank:K_0^\sht(X)_\Q\ra\Q$. Thus $\ch(\lb\O_X\rb-\lb L\rb)\in A^{\ge 1}(X,\Q)$ as $\rank\O_X=\rank L=1$, so that $\ch\bigl((\lb\O_X\rb-\lb L\rb\rb)^4\bigr)\in A^{\ge 4}(X,\Q)=0$. Therefore $(\lb\O_X\rb-\lb L\rb)^4=0$ in $K_0^\sht(X)_\Q$.

It follows from \eq{pt2eq8} that
\e
H_2(\M_0,\Q)\cong K_2^\sht(X)_\Q\op\La^2K_1^\sht(X)_\Q.
\label{pt3eq4}
\e
We claim that
\e
\sum_{i=0}^7(-1)^i\binom{7}{i}\bigl(J^L_*\bigr)^i=0:H_2(\M_0,\Q)\longra H_2(\M_0,\Q).
\label{pt3eq5}
\e
To see this, note that $J^L_*$ acts as multiplication by $\lb L\rb\in K_0^\sht(X)_\Q$ on the $K_2^\sht(X)_\Q$ factor in \eq{pt3eq4}, so $(\id-J^L_*)^7$ acts as multiplication by $(\lb\O_X\rb-\lb L\rb)^7=0$ in $K_0^\sht(X)_\Q$. Also $J^L_*$ acts as multiplication by $\lb L\rb\ot \lb L\rb$ in $K_0^\sht(X)_\Q\ot K_0^\sht(X)_\Q$ on the $\La^2K_1^\sht(X)_\Q$ factor in \eq{pt3eq4}. Using
\begin{equation*}
(1\!\ot\! 1\!-\!\lb L\rb\!\ot\! \lb L\rb)^7\!=\!\bigl((1\!-\!\lb L\rb)\!\ot\! 1+1\ot (1\!-\!\lb L\rb)\!-\!(1\!-\!\lb L\rb)\!\ot\!(1\!-\!\lb L\rb)\bigr)^7,
\end{equation*}
we can expand $(1\ot 1-\lb L\rb\ot\lb L\rb)^7$ as a sum of terms each of which contains a factor of $(1-\lb L\rb)^4$ acting on either the left or right factor of $K_1^\sht(X)_\Q$, so again $(\id-J^L_*)^7=0$ on this factor.

Now for any $n\in\Z$, applying \eq{pt3eq5} to $I_{e_\eta=0}([\M^\ss_{(\be,n)}(\mu^\om)]_\inv)^0$ and noting that $\bigl(J^L_*\bigr){}^i\bigl(I_{e_\eta=0}([\M^\ss_{(\be,n)}(\mu^\om)]_\inv)^0\bigr)\!=\!I_{e_\eta=0}\bigl([\M^\ss_{(\be,n+id_\be)}(\mu^\om)]_\inv\bigr){}^0$ by \eq{pt3eq3} gives
\e
\sum_{i=0}^7(-1)^i\binom{7}{i}I_{e_\eta=0}\bigl([\M^\ss_{(\be,n+id_\be)}(\mu^\om)]_\inv\bigr){}^0=0\quad\text{in $H_2(\M_0,\Q)$.}
\label{pt3eq6}
\e

Fix $1\le j\le d$ and consider the function $F_j:\Z\ra H_2(\M_0,\Q)$ mapping $F_j:k\mapsto I_{e_\eta=0}([\M^\ss_{(\be,j+kd_\be)}(\mu^\om)]_\inv)^0$. Then $\sum_{i=0}^7(-1)^i\binom{7}{i}F_j(k+i)=0$ for each $k\in\Z$ by \eq{pt3eq6}. This is a difference equation whose solutions are exactly polynomials $F_j(k)\in H_2(\M_0,\Q)[k]$ of degree $\le 6$. Part (a) follows.

For (b), fix $N\ge 0$, and use the notation of \S\ref{pt27}. Then $H_2^{G,\le N}(\M_0,\Q)$ has a finite filtration
\ea
H_2^{G,\le N}(\M_0,\Q)&=\ti F^0H_2^{G,\le N}(\M_0,\Q)\supseteq\ti F^{-1}H_2^{G,\le N}(\M_0,\Q)\supseteq\cdots
\nonumber\\
&\supseteq\ti F^{-N-1}H_2^{G,\le N}(\M_0,\Q)=0,
\label{pt3eq7}
\ea
where $\ti F^pH_2^{G,\le N}(\M_0,\Q)\!=\!F^pH_2^G(\M_0,\Q)/F^{-N-1}H_2^G(\M_0,\Q)$, and by \eq{pt2eq30}
\e
\frac{\ti F^pH_2^{G,\le N}(\M_0,\Q)}{\ti F^{p-1}H_2^{G,\le N}(\M_0,\Q)}\cong E^\iy_{p,2-p},\quad p=0,-1,\ldots,-N.
\label{pt3eq8}
\e

The entire spectral sequence $H^{-p}_G(*,\Q)\ot H_q(\M_0,\Q)\Ra H_{p+q}^G(\M_0,\Q)$ in \S\ref{pt27} is compatible with the $G$-equivariant action of $J^L:\M_0\ra\M_0$. Thus $J^L_*\!:\!H_n^G(\M_0,\Q)\!\ra\! H_n^G(\M_0,\Q)$ preserves the filtration $(F^pH_n^G(\M_0,\Q))_{p\le 0}$, and $J^L_*$ has actions on the $E^k_{p,q}$ commuting with the $d^k_{p,q}$ and compatible with the isomorphisms \eq{pt2eq29}--\eq{pt2eq30}.

Consider the action of $J^L_*$ on $E^2_{p,2-p}=H^{-p}_G(*,\Q)\ot_\Q H_{2-p}(\M_0)$ for $p=0,-1,\ldots,-N$. By the argument in (a) we can show that
\begin{equation*}
(\id-J^L_*)^{3(2-p)+1}=0:H_{2-p}(\M_0)\longra H_{2-p}(\M_0).
\end{equation*}
Hence we see that
\begin{equation*}
(\id-J^L_*)^{3(2+N)+1}=0:E^2_{p,2-p}\longra E^2_{p,2-p}, \quad p=0,-1,\ldots,-N.
\end{equation*}
Now $E^{k+1}_{p,2-p}$ is obtained by taking cohomology on $E^k_{p,2-p}$ for $k=2,3,\ldots,\iy$, so by induction we see that $(\id-J^L_*)^{3(2+N)+1}=0$ on $E^k_{p,2-p}$ for $k=2,3,\ldots,\iy$. Thus by \eq{pt3eq8} we see that
\begin{equation*}
(\id-J^L_*)^{3(2+N)+1}\bigl(\ti F^pH_2^{G,\le N}(\M_0,\Q)\bigr)\subseteq \ti F^{p-1}H_2^{G,\le N}(\M_0,\Q)
\end{equation*}
for $p=0,-1,\ldots,-N$. Applying this $N+1$ times and using \eq{pt3eq7}, we see that
\e
(\id-J^L_*)^{(N+1)(3(2+N)+1)}=0:H_2^{G,\le N}(\M_0,\Q)\longra H_2^{G,\le N}(\M_0,\Q).
\label{pt3eq9}
\e
The argument in (a), but using \eq{pt3eq9} rather than \eq{pt3eq5}, now proves~(b).
\end{proof}

\begin{rem}
\label{pt3rem1}
Note that the upper bound for $\deg P_j$ in Proposition \ref{pt3prop1}(b) goes to $\iy$ as $N\ra\iy$. The authors expect the analogue of Proposition \ref{pt3prop1}(b) for $I_{e_\eta=0}([\M^\ss_{(\be,n)}(\mu^\om)]^G_\inv)^0$ (which would in effect be the limit of Proposition \ref{pt3prop1}(b) as $N\ra\iy$) to be false in general. As Proposition \ref{pt3prop1}(b) is an essential ingredient in proving Theorem \ref{pt1thm2}(b), the authors also expect the analogue of Theorem \ref{pt1thm2}(b) for $[P_n^\alg(X,\be)]^{0,G}_\virt$ to be false in general. That is, $F^{G,\le N}(q)$ in Theorem \ref{pt1thm2}(b) may have poles at $q=e^{2\pi ik/d_\be}$ whose degree goes to infinity as $N\ra\iy$, so that $F^G(q)=\varprojlim_{N\ra\iy}F^{G,\le N}(q)$ is not a rational function.
\end{rem}

\subsection{Piecewise quasi-polynomial functions}
\label{pt32}

We develop some material we will need in \S\ref{pt33}.

\begin{dfn}
\label{pt3def1}
Let $V$ be a $\Q$-vector space. A function $f:\Z\ra V$ is a {\it quasi-polynomial} if there exist $d\ge 1$ and polynomials $P_1,\dots,P_d\in V[x]$ such that
\begin{equation*}
f(n)=P_a(n)\qquad\text{whenever }n\equiv a\mod d.
\end{equation*}
More generally, a function $f:\Z^k\ra V$ is a {\it quasi-polynomial\/} if there exist $d\ge 1$ and
polynomials $P_{\bs a}(x_1,\dots,x_k)\in V[x_1,\dots,x_k]$ for $\bs a\in\{1,2,\ldots,d\}^k$, such that
\begin{equation*}
f(n_1,\dots,n_k)=P_{\bs a}(n_1,\dots,n_k)
\quad\text{whenever }(n_1,\dots,n_k)\equiv \bs a\mod d.
\end{equation*}
Quasi-polynomials $f,g:\Z^k\ra V$ are closed under addition, that is, $f+g$ is also quasi-polynomial, taking lowest common multiples of periods.

A {\it chamber decomposition} of $\Z^k$ is a finite partition
\begin{equation*}
\Z^k=A_1\amalg\cdots\amalg A_N
\end{equation*}
in which each chamber $A_i$ is cut out by finitely many rational affine equalities and inequalities of the form
\e
C_1n_1+\cdots+C_kn_k\ \square\ D,
\quad \square\in\{<,\le,=,\ge,>\},
\quad C_1,\dots,C_k,D\in\Q.
\label{pt3eq10}
\e

A function $f:\Z^k\ra V$ is called {\it piecewise quasi-polynomial\/} if there is a chamber decomposition $\Z^k=A_1\amalg\cdots\amalg A_N$ and quasi-polynomials $g_1,\ldots,g_N:\Z^k\ra V$ such that $f\vert_{A_i}=g_i\vert_{A_i}$ for $i=1,\ldots,N$. Piecewise quasi-polynomial functions $f,g:\Z^k\ra V$ are closed under addition.
\end{dfn}

The next proposition is a standard result in Ehrhart theory. See, for example, Beck and Robins~\cite[Chs.~3--4]{BeRo}.

\begin{prop}
\label{pt3prop2}
Let $V$ be a $\Q$-vector space and\/ $F:\Z^k\ra V$ be piecewise quasi-polynomial. Assume that for each\/ $n\in\Z$ one has $F(n_1,\dots,n_k)=0$ for all but finitely many $(n_1,\dots,n_k)$ with\/ $n_1+\cdots+n_k=n$. Define $H:\Z\ra V$ by
\begin{equation*}
H(n)=\sum_{n_1+\cdots+n_k=n} F(n_1,\dots,n_k).
\end{equation*}
Then $H$ is piecewise quasi-polynomial in $n$.
\end{prop}

\subsection{\texorpdfstring{Dependence of $[P^\alg_n(X,\be)]_\virt$ on $n$}{Dependence of [Pₙᵃˡᵍ(X,β)]ᵥᵢᵣₜ on n}}
\label{pt33}

\begin{prop}
\label{pt3prop3}
{\bf(a)} Let\/ $\be\in A_1^\alg(X)$ be superpositive. As in {\rm\eq{pt2eq12},} regard\/ $[P_n^\alg(X,\be)]^0_\virt$ as an element of\/ $H_{2c_1(X)\cdot\be}(\M_0,\Q)$. Then the function $\Z\ra H_{2c_1(X)\cdot\be}(\M_0,\Q),$ $n\mapsto[P_n^\alg(X,\be)]^0_\virt$ is piecewise quasi-polynomial in\/~$n$.
\smallskip

\noindent{\bf(b)} Suppose a linear algebraic $\C$-group $G$ acts on $X,$ and acts trivially on $A_1^\alg(X),$ and\/ $N\ge 0$. Then the analogue of\/ {\bf(a)} holds for the function $\Z\ra H_{2c_1(X)\cdot\be}^{G,\le N}(\M_0,\Q)$ mapping $n\mapsto[P_n^\alg(X,\be)]_\virt^{0,G,\le N}$.
\end{prop}

\begin{proof}
For $\be\in A_1^\alg(X)$ superpositive, define $K(\be)\ge 1$ to be the maximum number $K$ such that we may write $\be=\be_1+\cdots+\be_K$ with $\be_i\in A_1^\alg(X)$ effective for $i=1,\ldots,K$. We will prove the proposition by induction on $K(\be)=1,2,\ldots.$ Our inductive hypothesis for $K=0,1,\ldots$ is:
\begin{itemize}
\setlength{\itemsep}{0pt}
\setlength{\parsep}{0pt}
\item[$(*)_K$] Suppose parts (a),(b) of the proposition hold whenever $K(\be)\le K$.
\end{itemize}

The first step $K=0$ is trivial. For the inductive step, suppose $(*)_K$ holds for some $K=0,1,\ldots,$ and let $\be\in A_1^\alg(X)$ be superpositive with $K(\be)=K+1$. Theorem \ref{pt2thm6} for this $\be$ gives $C_\be\in\R$ such that if $n>C_\be\,\om\cdot \Pi_\alg^\hom(\be)$ then equation \eq{pt2eq41} holds in the Lie algebra $H_*(\M^\pl,\Q)$. We claim that using Theorem \ref{pt2thm3}(d),(e) with $i=0$ and $\eta=1\in H_0(X,\Q)$ we can lift \eq{pt2eq41} to the following equation in $H_*(\M,\Q)$:
\ea
&[P^\alg_n(X,\be)]_\virt=
\label{pt3eq11}\\
&-\sum_{\begin{subarray}{l}
1\le j\le k,\; k\ge 2, \\
\be=\be_1+\cdots+\be_k,  \\
n=n_1+\cdots+n_k,\\ 
\text{$\be_i$ effective and} \\
\text{$n_i\in\Z,$ $i\ne j,$} \\
\text{either $\be_j$} \\
\text{effective and}\\
\text{$n_j\in\Z,$ or}\\
\text{$(\be_j,n_j)\!=\!(0,0)$}
\end{subarray}}\,\,
\begin{aligned}[t]&
\ti U\bigl((0,\be_1,n_1),\ldots,(0,\be_{j-1},n_{j-1}),(1,\be_j,n_j),\\
&(0,\be_{j+1},n_{j+1}),\ldots,(0,\be_k,n_k);\ac\mu^\om_\iy,\ac\mu^\om_{\mu^\om(\be,n)-\ep}\bigr)\cdot \\
&\bigl[\bigl[\cdots\bigl[I_{e_\eta=0}([\M^\ss_{(\be_1,n_1)}(\mu^\om)]_\inv),I_{e_\eta=0}([\M^\ss_{(\be_2,n_2)}(\mu^\om)]_\inv)\bigr]_{e_1=0},\\
&\ldots,I_{e_\eta=0}([\M^\ss_{(\be_{j-1},n_{j-1})}(\mu^\om)]_\inv)\bigr]_{e_1=0},[P^\alg_{n_j}(X,\be_j)]_\virt\bigr]_{e_1=0},\\
&I_{e_\eta=0}([\M^\ss_{(\be_{j+1},n_{j+1})}(\mu^\om)]_\inv)\bigr]_{e_1=0},\ldots,\\
&I_{e_\eta=0}([\M^\ss_{(\be_k,n_k)}(\mu^\om)]_\inv)\bigr]_{e_1=0}.
\end{aligned}
\nonumber
\ea

Here our convention is that $[P_n^\alg(X,\be)]_\virt\in H_{2c_1(X)\cdot\be}(\M_{\up(1,\be,n)},\Q)$ as in \eq{pt2eq10}, and similarly for $[P^\alg_{n_j}(X,\be_j)]_\virt$. Also $I_{e_\eta=0}([\M^\ss_{(\be_i,n_i)}(\mu^\om)]_\inv)\in H_2(\M_{\pi(\be_i,n_i)},\Q)$ as in \eq{pt2eq39}, for $\eta=\PD(\om)$ as in \S\ref{pt31}. Note that in \eq{pt3eq11} we do {\it not\/} yet transfer classes from $H_*(\M_\al,\Q)$ to $H_*(\M_0,\Q)$ using \eq{pt2eq7} as in \eq{pt2eq12} and \eq{pt3eq1}, which would be indicated by superscripts~0.

We have written the `Lie brackets' in \eq{pt3eq11} as $[\,,\,]_{{e_1=0}}$, to indicate that they are defined to be compatible with the morphisms $I_{e_1=0}$ in Theorem \ref{pt2thm3} for $i=0$ and $\eta=1$. We define $[\,,\,]_{{e_1=0}}$ as follows:
\begin{itemize}
\setlength{\itemsep}{0pt}
\setlength{\parsep}{0pt}
\item[(i)] If $u\in H_*(\M_{\up(0,\be',n')},\Q)$ and $v\in H_*(\M_{\up(0,\be'',n')},\Q)$ for $\be',\be''\in A_1^\alg(X)$ effective then $[u,v]_{{e_1=0}}=u_0(v)$, using the vertex algebra structure on $H_*(\M,\Q)$ in Theorem \ref{pt2thm1}.
\item[(ii)] If $u\in H_*(\M_{\up(1,\be',n')},\Q)$ and $v\in H_*(\M_{\up(0,\be'',n')},\Q)$ for $\be''$ effective then
\begin{equation*}
[u,v]_{{e_1=0}}=\sum_{k\ge 0}\frac{(-1)^k}{k!}D^k\bigl(u_k(v)\bigr).
\end{equation*}
\item[(iii)] If $u\in H_*(\M_{\up(0,\be',n')},\Q)$ and $v\in H_*(\M_{\up(1,\be'',n')},\Q)$ for $\be'$ effective then
\begin{equation*}
[u,v]_{{e_1=0}}=-\sum_{k\ge 0}\frac{(-1)^k}{k!}D^k\bigl(v_k(u)\bigr).
\end{equation*}
\end{itemize}
All brackets in \eq{pt3eq11} fall into one of these three cases.

Here is how to understand all this. We are lifting \eq{pt2eq41} from $H_*(\M^\pl,\Q)$ to $H_*(\M,\Q)$. For terms in \eq{pt2eq41} in $H_*(\M^\pl_{\up(0,\be',n')},\Q)$, such as $[\M^\ss_{(\be_i,n_i)}(\mu^\om)]_\inv$ and any repeated Lie brackets of these in \eq{pt2eq41}, we are happy with {\it any\/} lift from $H_*(\M^\pl_{\up(0,\be',n')},\Q)$ to $H_*(\M_{\up(0,\be',n')},\Q)$. For the $[\M^\ss_{(\be_i,n_i)}(\mu^\om)]_\inv$ we lift using $I_{e_\eta=0}$ for compatibility with Proposition \ref{pt3prop1} later, but the choice of lift does not affect the outcome. For (i) above we take the Lie bracket to be $u_0(v)$, as this is correct up to choice of lift by~\eq{pt2eq13}.

For terms in \eq{pt2eq41} in $H_*(\M^\pl_{\up(1,\be',n')},\Q)$, such as $\Pi^\pl_*([P^\alg_n(X,\be)]_\virt)$ and $\Pi^\pl_*([P^\alg_{n_j}(X,\be_j)]_\virt)$, and any repeated Lie bracket of $\Pi^\pl_*([P^\alg_{n_j}(X,\be_j)]_\virt)$ with multiple $[\M^\ss_{(\be_i,n_i)}(\mu^\om)]_\inv$'s, we lift to $H_*(\M_{\up(1,\be',n')},\Q)$ using $I_{e_1=0}$ in Theorem \ref{pt2thm3} for $i=0$ and $\eta=1$. Equation \eq{pt2eq26} shows that we can replace $I_{e_1=0}\ci\Pi^\pl_*([P_n^\alg(X,\be)]_\virt)$ by $[P_n^\alg(X,\be)]_\virt$, and similarly for $[P^\alg_{n_j}(X,\be_j)]_\virt$. The brackets $[\,,\,]_{e_1=0}$ in (ii),(iii) above are now justified by Theorem \ref{pt2thm3}(e), which also shows that we can use arbitrary lifts from $H_*(\M^\pl_{\up(0,\be',n')},\Q)$ to $H_*(\M_{\up(0,\be',n')},\Q)$, since \eq{pt2eq25} is independent of the choice of lift $v'$. This proves equation~\eq{pt3eq11}.

Next, we wish to rewrite \eq{pt3eq11} solely in terms of classes in $H_*(\M_0,\Q)$, such as $[P_n^\alg(X,\be)]^0_\virt$ in \eq{pt2eq12} and $I_{e_\eta=0}([\M^\ss_{(\be_i,n_i)}(\mu^\om)]_\inv)^0$ in \eq{pt3eq1}, using the isomorphisms $H_*(\M_{\up(d',\be',n')},\Q)\cong H_*(\M_0,\Q)$ in \eq{pt2eq7}. For classes $(d',\be',n')$ and $(d'',\be'',n'')$ as in (i)--(iii) above, with $(d',d'')\in\{(0,0),(1,0),(0,1)\}$, define a $\Q$-bilinear bracket $[\,,\,]^{(d',\be',n')}_{(d'',\be'',n'')}$ on $H_*(\M_0,\Q)$ by the commutative diagram
\e
\begin{gathered}
\xymatrix@C=150pt@R=13pt{ *+[r]{H_*(\M_0,\Q)\t H_*(\M_0,\Q)} \ar[r]_(0.6){[\,,\,]^{(d',\be',n')}_{(d'',\be'',n'')}} \ar[d]^(0.4){\eq{pt2eq7}\t\eq{pt2eq7}} & *+[l]{H_*(\M_0,\Q)}  \\
*+[r]{\begin{subarray}{l}\ts H_*(\M_{\up(d',\be',n')},\Q)\t \\ \ts H_*(\M_{\up(d'',\be'',n'')},\Q)\end{subarray}} \ar[r]^(0.4){[\,,\,]_{{e_1=0}}} &  *+[l]{H_*(\M_{\up(d'+d'',\be'+\be'',n'+n'')},\Q).\!} \ar[u]^{\eq{pt2eq7}} }
\end{gathered}
\label{pt3eq12}
\e

Equation \eq{pt3eq11} is now equivalent to
\ea
&[P^\alg_n(X,\be)]^0_\virt=
\label{pt3eq13}\\
&-\!\!\!\!\!\!\sum_{\begin{subarray}{l}
1\le j\le k,\; k\ge 2, \\
\be=\be_1+\cdots+\be_k,  \\
n=n_1+\cdots+n_k,\\ 
\text{$\be_i$ effective and} \\
\text{$n_i\in\Z,$ $i\ne j,$} \\
\text{either $\be_j$} \\
\text{effective and}\\
\text{$n_j\in\Z,$ or}\\
\text{$(\be_j,n_j)\!=\!(0,0)$}
\end{subarray}}\,\,
\begin{aligned}[t]&
\ti U\bigl((0,\be_1,n_1),\ldots,(0,\be_{j-1},n_{j-1}),(1,\be_j,n_j),\\
&(0,\be_{j+1},n_{j+1}),\ldots,(0,\be_k,n_k);\ac\mu^\om_\iy,\ac\mu^\om_{\mu^\om(\be,n)-\ep}\bigr)\cdot \bigl[\bigl[\cdots \\
&\bigl[I_{e_\eta=0}([\M^\ss_{(\be_1,n_1)}(\mu^\om)]_\inv)^0,I_{e_\eta=0}([\M^\ss_{(\be_2,n_2)}(\mu^\om)]_\inv)^0\bigr]^{(0,\be_1,n_1)}_{(0,\be_2,n_2)},\\
&\ldots,I_{e_\eta=0}([\M^\ss_{(\be_{j-1},n_{j-1})}(\mu^\om)]_\inv)^0\bigr]^{(0,\sum_{i=1}^{j-2}\be_i,\sum_{i=1}^{j-2}n_i)}_{(0,\be_{j-1},n_{j-1})},\\
&[P^\alg_{n_j}(X,\be_j)]^0_\virt\bigr]^{(0,\sum_{i=1}^{j-1}\be_i,\sum_{i=1}^{j-1}n_i)}_{(1,\be_j,n_j)},\\
&I_{e_\eta=0}([\M^\ss_{(\be_{j+1},n_{j+1})}(\mu^\om)]_\inv)^0\bigr]^{(1,\sum_{i=1}^j\be_i,\sum_{i=1}^jn_i)}_{(0,\be_{j+1},n_{j+1})},\ldots,\\
&I_{e_\eta=0}([\M^\ss_{(\be_k,n_k)}(\mu^\om)]_\inv)^0\bigr]^{(1,\sum_{i=1}^{k-1}\be_i,\sum_{i=1}^{k-1}n_i)}_{(0,\be_k,n_k)}.
\end{aligned}
\nonumber
\ea

In the sum on the right hand side, there are finitely many possibilities for $j,k$ and $\be_1,\ldots,\be_k$. Fix $j,k,\be_1,\ldots,\be_k$, and consider how the term in the sum behaves as a function of $(n_1,\ldots,n_k)\in\Z^k$. By Proposition \ref{pt3prop1}, for $i=1,\ldots,k$ with $i\ne j$, the function $n_i\mapsto I_{e_\eta=0}([\M^\ss_{(\be_i,n_i)}(\mu^\om)]_\inv)^0$ is quasi-polynomial in $n_i$. For the term $[P^\alg_{n_j}(X,\be_j)]^0_\virt$ on the right hand side, if $\be_j\ne 0$ then $\be_j$ can be split into $K(\be_j)\ge 1$ effective summands, so $\be$ can be split into $K(\be_j)+k-1$ effective summands. Thus $K(\be_j)+k-1\le K+1$, giving $K(\be_j)\le K$ as $k\ge 2$. So by the inductive hypothesis $(*)_K$, the function $n_j\mapsto [P^\alg_{n_j}(X,\be_j)]^0_\virt$ is piecewise quasi-polynomial in $n_j$.

Consider the functions $\Z^k\ra\Q$ mapping
\begin{align*}
(n_1,\ldots,n_k)&\longmapsto U\bigl((0,\be_1,n_1),\ldots,(0,\be_{j-1},n_{j-1}),(1,\be_j,n_j),(0,\be_{j+1},n_{j+1}),\\
&\qquad\qquad\qquad \ldots,(0,\be_k,n_k);\ac\mu^\om_\iy,\ac\mu^\om_{\mu^\om(\be,n_1+\cdots+n_k)-\ep}\bigr),\\
(n_1,\ldots,n_k)&\longmapsto \ti U\bigl((0,\be_1,n_1),\ldots,(0,\be_{j-1},n_{j-1}),(1,\be_j,n_j),(0,\be_{j+1},n_{j+1}),\\
&\qquad\qquad\qquad \ldots,(0,\be_k,n_k);\ac\mu^\om_\iy,\ac\mu^\om_{\mu^\om(\be,n_1+\cdots+n_k)-\ep}\bigr).
\end{align*}
By Definition \ref{pt2def11} for $U(\cdots)$ and Definitions \ref{pt2def12} and \ref{pt2def13} for $\mu^\om$ and $\ac\mu^\om$, and because we have chosen $\om\in H^2(X,\Q)$ and $\ep>0$ in $\Q$, we see that the first function depends on finitely many rational affine linear equalities and inequalities in $n_1,\ldots,n_k$. Thus there exists a chamber decomposition $\Z^k=A_1\amalg\cdots\amalg A_N$ as in Definition \ref{pt3def1} such that the first function is constant on each chamber $A_i$.

Now the coefficients $\ti U(\cdots)$ are characterized in Theorem \ref{pt2thm4}, and they are {\it not uniquely defined}. Refining the decomposition $\Z^k=A_1\amalg\cdots\amalg A_N$ if necessary, we can suppose that if we replace the input $(0,\be_1,n_1),\ldots,(0,\be_k,n_k)$ in $U(\cdots)$ in the first function by any permutation of $(0,\be_1,n_1),\ldots,(0,\be_k,n_k)$, then the modified function is still constant on each $A_i$. Then by Theorem \ref{pt2thm4}, we can {\it choose\/} the $\ti U(\cdots)$ such that the second function is constant on each $A_i$, as the defining property of the $\ti U(\cdots)$ depends only on $U(\cdots)$ for all permutations of~$(0,\be_1,n_1),\ldots,(0,\be_{j-1},n_{j-1}),(1,\be_j,n_j),(0,\be_{j+1},n_{j+1}),\ldots,(0,\be_k,n_k)$.

Next fix $d',\be',d'',\be'',u,v$, and consider the function $\Z^2\ra H_*(\M_0,\Q)$ mapping $(n',n'')\mapsto [u,v]^{(d',\be',n')}_{(d'',\be'',n'')}$, for $[\,,\,]^{(d',\be',n')}_{(d'',\be'',n'')}$ as in \eq{pt3eq12}. From (i)--(iii) above and the definitions \eq{pt2eq15}--\eq{pt2eq16} of $D$ and $u_k(v)$, we can show that $D$ when applied to $D^j(u_k(v))$ is linear in $n'+n''$, and $u_k(v)$ is polynomial in $(n',n'')$. It is important in proving this that the factor $z^{\chi(\al,\be)+\chi(\be,\al)}$ in \eq{pt2eq16} with $\al=\up(d',\be',n')$ and $\be=\up(d'',\be'',n'')$ is independent of $n',n''$, so $n',n''$ do not change the powers of $z$ we need to take coefficients of. Elsewhere in \eq{pt2eq16}, after identifying $H_*(\M_\al,\Q)\cong H_*(\M_0,\Q)$ and $H_*(\M_\be,\Q)\cong H_*(\M_0,\Q)$ using \eq{pt2eq7}, the term $(\Psi_\al)_*(t^j\bt u)$ may be understood as polynomial of degree $j$ in $n'$, and the factor $-\cap z^ic_i\bigl((\cExt_{\al,\be}^\bu)^\vee\op\si_{\al,\be}^*(\cExt^\bu_{\be,\al})\bigr)$ as polynomial of degree $i$ in~$n',n''$.

Combining all the above, we see that for fixed $j,k,\be_1,\ldots,\be_k$, the term in the sum in \eq{pt3eq13} is piecewise quasi-polynomial in $(n_1,\ldots,n_k)\in\Z^k$, because it begins with inputs $I_{e_\eta=0}([\M^\ss_{(\be_i,n_i)}(\mu^\om)]_\inv)^0, P^\alg_{n_j}(X,\be_j)]^0_\virt$ which are (piecewise) quasi-polynomial in $n_i,n_j$, then modifies them by operations $[\,,\,]^{(d',\be',n')}_{(d'',\be'',n'')}$ which are polynomial in $n_1,\ldots,n_k$, and then multiplies them by coefficients $\ti U(\cdots)$ which are piecewise constant in $n_1,\ldots,n_k$. Therefore Proposition \ref{pt3prop2} shows that for fixed $j,k,\be_1,\ldots,\be_k$, the sum over $n=n_1+\cdots+n_k$ in \eq{pt3eq13} is piecewise quasi-polynomial in~$n$.

Summing over the finitely many possibilities for $j,k,\be_1,\ldots,\be_k$, we now see that if $n>C_\be\,\om\cdot \Pi_\alg^\hom(\be)$ then the function $n\mapsto [P^\alg_n(X,\be)]_\virt$ is piecewise quasi-polynomial in $n$. Now $P^\alg_n(X,\be)=\es$ for $n\le M$ for some $M\in\Z$. Thus we may make a finite chamber decomposition of $\Z$ which for $n>C_\be\,\om\cdot \Pi_\alg^\hom(\be)$ is the chamber decomposition for which $n\mapsto [P^\alg_n(X,\be)]_\virt$ is quasi-polynomial on each chamber, together with a chamber $\{n\in\Z:n\le M\}$ on which $[P^\alg_n(X,\be)]_\virt=0$, and finitely many singleton chambers $\{n\}$ for $M<n\le C_\be\,\om\cdot \Pi_\alg^\hom(\be)$. Then $n\mapsto [P^\alg_n(X,\be)]_\virt$ is quasi-polynomial on each of these chambers, so $n\mapsto [P^\alg_n(X,\be)]_\virt$ is piecewise quasi-polynomial on all of $\Z$. This proves Proposition \ref{pt3prop3}(a) for~$\be$.

Proposition \ref{pt3prop3}(b) for $\be$ in truncated $G$-equivariant homology $H_*^{G,\le N}(\cdots)$ works by essentially the same argument. It is important that $H_k^{G,\le N}(\cdots)=0$ for $k<-N$, which means that the sums involved are finite, and the degrees of the polynomials bounded above; the same proof would not work for~$H_*^G(\cdots)$.

This completes the inductive step, so Proposition \ref{pt3prop3} holds by induction.
\end{proof}

\subsection{Proof of Theorem \ref{pt1thm2}}
\label{pt34}

For Theorem \ref{pt1thm2}(a), by Proposition \ref{pt3prop3}(a) the function $n\mapsto[P_n^\alg(X,\be)]^0_\virt$ is piecewise quasi-polynomial in $n$, relative to some chamber decomposition of $\Z$. We may take this chamber decomposition to have two infinite chambers $\{n\in\Z:n\le M\}$ and $\{n\in\Z:n\ge N\}$ for $M\le N$ in $\Z$, and other finite chambers in $M<n<N$, where $[P_n^\alg(X,\be)]^0_\virt=0$ for $n\le M$. For the chamber $\{n\in\Z:n\ge N\}$ there exist $d\ge 1$ and polynomials $P_j\in H_{2c_1(X)\cdot\be}(\M_0,\Q)[n]$ for $1\le j\le d$ such that $[P_n^\alg(X,\be)]^0_\virt=P_j(n)$ if $n\ge N$ and $n\equiv j\mod d$. It is easy to show that expanding $(1-q^d)^{-k}$ in powers of $(q^d)^{\ge 0}$ we have
\begin{equation*}
\sum_{n\ge N:n\equiv j\!\!\!\!\mod d}q^nP_j(n)=\frac{q^jQ_j(q^d)}{(1-q^d)^{\deg P_j+1}},
\end{equation*}
where $Q_j(q^d)$ is a Laurent polynomial in $q^d$. Hence
\begin{equation*}
\sum_{n\in\Z}[P_n^\alg(X,\be)]^0_\virt q^n=\sum_{n=M+1}^{N-1}[P_n^\alg(X,\be)]^0_\virt q^n
+\sum_{j=1}^d\frac{q^jQ_j(q^d)}{(1-q^d)^{\deg P_j+1}}.
\end{equation*}
The right hand side is a rational function $F(q)\in H_{2c_1(X)\cdot\be}(\M_0,\Q)(q),$ which has poles only at $q=0$ and $q=e^{2\pi ik/d}$ for $k=1,\ldots,d$. This proves Theorem \ref{pt1thm2}(a). Part (b) is proved in the same way, using Proposition~\ref{pt3prop3}(b).

\medskip

\noindent{\small\sc Reginald Anderson, Department of Mathematics,
University of California, Irvine.

\noindent E-mail: {\tt reginala@uci.edu.}

\noindent Dominic Joyce, The Mathematical Institute, Radcliffe
Observatory Quarter, Woodstock Road, Oxford, OX2 6GG, U.K. 

\noindent E-mail:  {\tt dominic.joyce@maths.ox.ac.uk.}

\smallskip

}


\begin{thebibliography}{200}
\addcontentsline{toc}{section}{References}

\bibitem{AOV} D. Abramovich, M. Olsson, and A. Vistoli, {\it Tame stacks in positive characteristic}, Ann. Inst. Fourier 58 (2008), 1057--1091. \href{https://arxiv.org/abs/math/0703310}{math.AG/0703310}.

\bibitem{Ande} R. Anderson, {\it Examples of descendent generating series for Pand\-hari\-pan\-de--Thomas stable pairs on smooth projective Fano threefolds via one-dimensional wall-crossing}, preprint, 2026.

\bibitem{AnHe} B. Antieau and J. Heller, {\it Some remarks on topological K-theory of dg categories}, Proc. A.M.S. 146 (2018), 4211--421. \href{https://arxiv.org/abs/1709.01587}{arXiv:1709.01587}.

\bibitem{BeRo} M. Beck and S. Robins, {\it Computing the continuous discretely: integer-point enumeration in polyhedra}, 2nd edition, Springer, New York, 2015.

\bibitem{BeFa} K. Behrend and B. Fantechi, {\it The intrinsic normal cone}, Invent. Math. 128 (1997), 45--88. \href{https://arxiv.org/abs/alg-geom/9601010}{alg-geom/9601010}.

\bibitem{Blan} A. Blanc, {\it Topological K-theory of complex noncommutative spaces}, Compos. Math. 152 (2016), 489--555. \href{https://arxiv.org/abs/1211.7360}{arXiv:1211.7360}.

\bibitem{Borc} R.E. Borcherds, {\it Vertex algebras, Kac--Moody algebras, and the Monster}, Proc. Nat. Acad. Sci. U.S.A. 83 (1986), 3068--3071.

\bibitem{Brid} T. Bridgeland, {\it Hall algebras and curve-counting invariants}, J. A.M.S. 24 (2011), 969--998. \href{https://arxiv.org/abs/1002.4374}{arXiv:1002.4374}.

\bibitem{Brio} M. Brion, {\it Algebraic group actions on normal varieties}, Trans. Moscow Math. Soc. 78 (2017), 85--107. \href{https://arxiv.org/abs/1703.09506}{arXiv:1703.09506}.

\bibitem{FHW} E.M. Friedlander, C. Haesemeyer, and M.E. Walker, {\it Techniques, computations, and conjectures for semi-topological K-theory}, Math. Ann. 330 (2004), 759--807.

\bibitem{FrWa1} E.M. Friedlander and M.E. Walker, {\it Comparing K-theories for complex varieties}, Amer. J. Math. 123 (2001), 779--810. 

\bibitem{FrWa2} E.M. Friedlander and M.E.Walker, {\it Semi-topological K-theory using function complexes}, Topology 41 (2002), 591--644.

\bibitem{FrWa3} E.M. Friedlander and M.E. Walker, {\it Semi-topological K-theory}, pages 877--924 in {\it Handbook of K-theory}, Springer, Berlin, 2005.

\bibitem{FrBZ} E. Frenkel and D. Ben-Zvi, {\it Vertex algebras and algebraic curves}, A.M.S, Providence, RI, 2004.

\bibitem{GeMa} S.I. Gelfand and Y.I. Manin, {\it Methods of Homological Algebra}, second edition, Springer, 2002. 

\bibitem{Gies} D. Gieseker, {\it On the moduli of vector bundles on an algebraic surface}, Ann. Math. 106 (1977), 45--60.

\bibitem{Gome} T.L. G\'omez, {\it  Algebraic stacks}, Proc. Indian Acad. Sci. Math. Sci. 111 (2001), 1--31. \href{https://arxiv.org/abs/math/9911199}{math.AG/9911199}.

\bibitem{Gros1} J. Gross, {\it The homology of moduli stacks of complexes}, \href{https://arxiv.org/abs/1907.03269}{arXiv:1907.03269}, 2019.

\bibitem{Gros2} J. Gross, {\it Moduli spaces of complexes of coherent sheaves}, Oxford PhD thesis, 2020, available at \href{https://ora.ox.ac.uk/objects/uuid:857c53a5-345b-4ab9-9420-f94c8030b4b3}{ora.ox.ac.uk}.

\bibitem{GJT} J. Gross, D. Joyce and Y. Tanaka, {\it Universal structures in $\C$-linear enumerative invariant theories}, SIGMA 18 (2022), 068. \href{https://arxiv.org/abs/2005.05637}{arXiv:2005.05637}.

\bibitem{Hart} R. Hartshorne, {\it Algebraic Geometry}, Graduate
Texts in Math. 52, \hfil\break  Springer, New York, 1977.

\bibitem{Huyb} D. Huybrechts, {\it Fourier--Mukai transforms in Algebraic Geometry}, Oxford University Press, Oxford, 2006.

\bibitem{HuLe} D. Huybrechts and M. Lehn, {\it The geometry of moduli spaces of sheaves}, second edition, CUP, Cambridge, 2010. 

\bibitem{Joyc1} D. Joyce, {\it Configurations in abelian categories. I. Basic properties and moduli stacks}, Adv. Math. 203 (2006), 194--255. \href{https://arxiv.org/abs/math/0312190}{math.AG/0312190}.

\bibitem{Joyc2} D. Joyce, {\it Configurations in abelian categories. II. Ringel--Hall algebras}, Adv. Math. 210 (2007), 635--706. \href{https://arxiv.org/abs/math/0503029}{math.AG/0503029}.

\bibitem{Joyc3} D. Joyce, {\it Configurations in abelian categories. III. Stability conditions and identities}, Adv. Math. 215 (2007), 153--219. \href{https://arxiv.org/abs/math/0410267}{math.AG/0410267}.

\bibitem{Joyc4} D. Joyce, {\it Configurations in abelian categories. IV. Invariants and changing stability conditions}, Adv. Math. 217 (2008), 125--204. \hfil\break \href{https://arxiv.org/abs/math/0410268}{math.AG/0410268}.

\bibitem{Joyc5} D. Joyce, {\it Vertex algebra and Lie algebra structures on the homology of moduli spaces}, in preparation, 2026. Preliminary version (2018) available as {\it Ringel--Hall style vertex algebra and Lie algebra structures on the homology of moduli spaces} at \url{https://people.maths.ox.ac.uk/~joyce}.

\bibitem{Joyc6} D. Joyce, {\it Enumerative invariants and wall-crossing formulae in abelian categories}, \href{https://arxiv.org/abs/2111.04694}{arXiv:2111.04694}, version 2 in preparation, 2026.

\bibitem{JoSo} D. Joyce and Y. Song, {\it A theory of generalized Donaldson--Thomas invariants}, Mem. A.M.S. 217 (2012), no. 1020. \href{https://arxiv.org/abs/0810.5645}{arXiv:0810.5645}.

\bibitem{KaMo1} I. Karpov and M. Moreira, {\it Generalized K-theoretic invariants and wall-crossing via non-abelian localization}, \href{https://arxiv.org/abs/2512.22360}{arXiv:2512.22360}, 2025.

\bibitem{KaMo2} I. Karpov and M. Moreira, {\it Rationality and symmetry of stable pairs generating series of Fano\/ $3$-folds}, \href{https://arxiv.org/abs/2604.06023}{arXiv:2604.06023}, 2026.

\bibitem{KoSo} M. Kontsevich and Y. Soibelman, {\it Stability structures, motivic Donaldson--Thomas invariants and cluster transformations}, \href{https://arxiv.org/abs/0811.2435}{arXiv:0811.2435}, 2008.

\bibitem{LaMo} G. Laumon and L. Moret-Bailly, {\it  Champs alg\'ebriques}, Ergeb. der Math. und ihrer Grenzgebiete 39, Springer, Berlin, 2000.

\bibitem{MNOP1} D. Maulik, N. Nekrasov, A. Okounkov, and R. Pandharipande, {\it Gromov--Witten theory and Donaldson--Thomas theory. I}, Compos. Math. 142 (2006), 1263--1285. \href{https://arxiv.org/abs/math/0312059}{math.AG/0312059}.

\bibitem{MNOP2} D. Maulik, N. Nekrasov, A. Okounkov, and R. Pandharipande, {\it Gromov--Witten theory and Donaldson--Thomas theory. II}, Compos. Math. 142 (2006), 1286--1304. \href{https://arxiv.org/abs/math/0406092}{math.AG/0406092}.

\bibitem{MOOP} D. Maulik, A. Oblomkov, A. Okounkov and R. Pandharipande, {\it Virasoro constraints for stable pairs on toric\/ $3$-folds}, Forum Math. Pi 10 (2022), no. e20. \href{https://arxiv.org/abs/2008.12514}{arXiv:2008.12514}.

\bibitem{McCl} J. McCleary, {\it A user's guide to spectral sequences}, Camb. Stud. Adv. Math. 58, CUP, Cambridge, 2001.

\bibitem{MiMo} J. Milnor and J. Moore, {\it On the structure of Hopf algebras}, Ann. Math. 81 (1965), 211--264.

\bibitem{Olss} M. Olsson, {\it  Algebraic Spaces and Stacks}, A.M.S. Colloquium Publications 62, A.M.S., Providence, RI, 2016.

\bibitem{Pand} R. Pandharipande, {\it Descendents for stable pairs on\/ $3$-folds}, pages 251--287 in Proc. Sympos. Pure Math. 99, A.M.S., 2018. \href{https://arxiv.org/abs/1703.01747}{arXiv:1703.01747}.

\bibitem{PaPi1} R. Pandharipande and A. Pixton, {\it Descendent theory for stable pairs on toric\/ $3$-folds}, J. Math. Soc. Japan 65 (2013), 1337--1372. \href{https://arxiv.org/abs/1011.4054}{arXiv:1011.4054}.

\bibitem{PaPi2} R. Pandharipande and A. Pixton, {\it Gromov--Witten/pairs correspondence for the quintic\/ $3$-fold}, J. A.M.S. 30 (2017), 389--449. \href{https://arxiv.org/abs/1206.5490}{arXiv:1206.5490}.

\bibitem{PaTh1} R. Pandharipande and R.P. Thomas, {\it Curve counting via stable pairs in the derived category}, Invent. Math. 178 (2009), 407--447. \href{https://arxiv.org/abs/0707.2348}{arXiv:0707.2348}.

\bibitem{PaTh2} R. Pandharipande and R.P. Thomas, {\it The\/ $3$-fold vertex via stable pairs}, Geom. Topol. 13 (2009), 1835--1876. \href{https://arxiv.org/abs/0709.3823}{arXiv:0709.3823}.

\bibitem{PaTh3} R. Pandharipande and R.P. Thomas, {\it Stable pairs and BPS invariants}, J. A.M.S. 23 (2010), 267--297. \href{https://arxiv.org/abs/0711.3899}{arXiv:0711.3899}.

\bibitem{PaTh4} R. Pandharipande and R.P. Thomas, {\it $13/2$ ways of counting curves}, pages 282--333 in L. Brambila-Paz et al., editors, {\it Moduli spaces}, L.M.S. Lecture Notes 411, CUP, 2014. \href{https://arxiv.org/abs/1111.1552}{arXiv:1111.1552}.

\bibitem{Pard} J. Pardon, {\it Universally counting curves in Calabi--Yau threefolds},  \hfil\break \href{https://arxiv.org/abs/2308.02948}{arXiv:2308.02948}, 2023.

\bibitem{Roma} M. Romagny, {\it Group actions on stacks and applications}, Michigan Math. J. 53 (2005), 209--236.

\bibitem{Ruda} A. Rudakov, {\it Stability for an abelian category}, J. Algebra 197 (1997), 231--245.

\bibitem{Simp} C. Simpson, {\it The topological realization of a simplicial presheaf}, \href{https://arxiv.org/abs/q-alg/9609004}{q-alg/9609004}, 1996.

\bibitem{StTh} J. Stoppa and  R.P. Thomas, {\it Hilbert schemes and stable pairs: GIT and derived category wall crossings}, Bull. Soc. Math. France 139 (2011), 297--339. \href{https://arxiv.org/abs/0903.1444}{arXiv:0903.1444}.

\bibitem{Thom} R.P. Thomas, {\it A holomorphic Casson invariant for Calabi--Yau $3$-folds, and bundles on $K3$ fibrations}, J. Diff. Geom. 54 (2000), 367--438. \hfil\break \href{https://arxiv.org/abs/math/9806111}{math.AG/9806111}.

\bibitem{Toda1} Y. Toda, {\it Limit stable objects on Calabi-Yau\/ $3$-folds}, Duke Math. J. 149 (2009), 157--208. \href{https://arxiv.org/abs/0803.2356}{arXiv:0803.2356}.

\bibitem{Toda2} Y. Toda, {\it Generating functions of stable pair invariants via wall-crossings in derived categories}, Adv. Stud. Pure Math. 59 (2010), 389--434. \hfil\break \href{https://arxiv.org/abs/0806.0062}{arXiv:0806.0062}.

\bibitem{Toda3} Y. Toda, {\it Curve counting theories via stable objects I. DT/PT correspondence}, J. A.M.S. 23 (2010), 1119--1157. \href{https://arxiv.org/abs/0902.4371}{arXiv:0902.4371}.

\bibitem{Toda4} Y. Toda, {\it Stability conditions and curve counting invariants on Calabi--Yau $3$-folds}, Kyoto J. Math. 52 (2012), 1--50. \href{https://arxiv.org/abs/1103.4229}{arXiv:1103.4229}.

\bibitem{Toen1} B. To\"en, {\it Higher and derived stacks: a global overview}, pages 435--487 in {\it Algebraic geometry -- Seattle 2005}. Proc. Sympos. Pure Math. 80, Part 1, A.M.S., 2009. \href{https://arxiv.org/abs/math/0604504}{math.AG/0604504}. 

\bibitem{Toen2} B. To\"en, {\it Derived Algebraic Geometry}, EMS Surveys in Mathematical Sciences 1 (2014), 153--240. \href{https://arxiv.org/abs/1401.1044}{arXiv:1401.1044}.

\bibitem{ToVa} B. To\"en and M. Vaqui\'e, {\it Moduli of objects in dg-categories}, Ann. Sci. \'Ec. Norm. Sup. 40 (2007), 387--444. \href{https://arxiv.org/abs/math/0503269}{math.AG/0503269}.

\bibitem{ToVe1} B. To\"en and G. Vezzosi, {\it  From HAG to DAG: derived moduli stacks}, pages 173--216 in {\it  Axiomatic, enriched and motivic homotopy theory}, NATO Sci. Ser. II Math. Phys. Chem., 131, Kluwer, Dordrecht, 2004. \href{https://arxiv.org/abs/math/0210407}{math.AG/0210407}.

\bibitem{ToVe2} B. To\"en and G. Vezzosi, {\it Homotopical Algebraic Geometry II: Geometric Stacks and Applications}, Mem. A.M.S. 193 (2008), no. 902. \hfil\break \href{https://arxiv.org/abs/math/0404373}{math.AG/0404373}.

\bibitem{Upme} M. Upmeier, {\it Homological Lie brackets on moduli spaces and pushforward operations in twisted K-theory}, J. Topol. 18 (2025), No. e70025. \href{https://arxiv.org/abs/2101.10990}{arXiv:2101.10990}.

\end{thebibliography}
\end{document}